\documentclass[a4paper]{amsart}

     \usepackage{amssymb}
     \usepackage{graphicx,color}
     
     \usepackage[normalem]{ulem}
     
     \usepackage{color}
     \usepackage{amsxtra}
     \usepackage{amssymb}
     \usepackage{mathtools}
     \usepackage{enumerate}
     \usepackage{nicefrac}
     \mathtoolsset{showonlyrefs} 
     \usepackage{ mathrsfs }
     \usepackage{tikz}
     \usetikzlibrary{positioning, calc}
     \usetikzlibrary{shadings}
     \usepackage{color}
     \usepackage{hyperref}
     \hypersetup{
       colorlinks   = true, 
       urlcolor     = blue, 
       linkcolor    = blue, 
       citecolor   = red 
     }
       
     \newtheorem{theorem}{Theorem}[section]

     \newtheorem{corollary}{Corollary}[section]
     
     \theoremstyle{definition}
     \newtheorem{definition}{Definition}[section]
     \newtheorem{remark}{Remark}[section]

     \numberwithin{equation}{section}
     
     \newcommand{\R}{{\mathbb R}}

     \newcommand{\RR}{{\mathbb R}}
     
     \newcommand{\N}{{\mathbb N}}

      \def\1{\raisebox{2pt}{\rm{$\chi$}}}

     \usepackage{enumerate}

     \makeatletter
     \newcommand{\subalign}[1]{%
       \vcenter{%
         \Let@ \restore@math@cr \default@tag
         \baselineskip\fontdimen10 \scriptfont\tw@
         \advance\baselineskip\fontdimen12 \scriptfont\tw@
         \lineskip\thr@@\fontdimen8 \scriptfont\thr@@
         \lineskiplimit\lineskip
         \ialign{\hfil$\m@th\scriptstyle##$&$\m@th\scriptstyle{}##$\hfil\crcr
           #1\crcr
         }%
       }%
     }
     \makeatother

\def\vint_#1{\mathchoice%
          {\mathop{\kern 0.2em\vrule width 0.6em height 0.69678ex depth -0.58065ex
                  \kern -0.8em \intop}\nolimits_{\kern -0.4em#1}}%
          {\mathop{\kern 0.1em\vrule width 0.5em height 0.69678ex depth -0.60387ex
                  \kern -0.6em \intop}\nolimits_{#1}}%
          {\mathop{\kern 0.1em\vrule width 0.5em height 0.69678ex
              depth -0.60387ex
                  \kern -0.6em \intop}\nolimits_{#1}}%
          {\mathop{\kern 0.1em\vrule width 0.5em height 0.69678ex depth -0.60387ex
                  \kern -0.6em \intop}\nolimits_{#1}}}
\def\vintslides_#1{\mathchoice%
          {\mathop{\kern 0.1em\vrule width 0.5em height 0.697ex depth -0.581ex
                  \kern -0.6em \intop}\nolimits_{\kern -0.4em#1}}%
          {\mathop{\kern 0.1em\vrule width 0.3em height 0.697ex depth -0.604ex
                  \kern -0.4em \intop}\nolimits_{#1}}%
          {\mathop{\kern 0.1em\vrule width 0.3em height 0.697ex depth -0.604ex
                  \kern -0.4em \intop}\nolimits_{#1}}%
          {\mathop{\kern 0.1em\vrule width 0.3em height 0.697ex depth -0.604ex
                  \kern -0.4em \intop}\nolimits_{#1}}}

     \def\undertilde#1{\mathord{\vtop{\ialign{##\crcr
     $\hfil\displaystyle{#1}\hfil$\crcr\noalign{\kern1.5pt\nointerlineskip}
     $\hfil\tilde{}\hfil$\crcr\noalign{\kern1.5pt}}}}}

     \DeclareMathAlphabet{\mathpzc}{OT1}{pzc}{m}{it}
     \newcommand{\dist}{{\mathpzc{d}}}
     \title[The trace fractional Laplacian]{
     The trace fractional Laplacian and the mid-range fractional Laplacian
     }

\author[J. D. Rossi]{Julio D. Rossi}
\address{J. D. Rossi. Departamento de Matem\'atica, FCEyN, Universidad de Buenos Aires,
 Pabell\'on I, Ciudad Universitaria (1428),
Buenos Aires, Argentina.}
\email{jrossi@dm.uba.ar}

   \author[J. Ruiz-Cases]{Jorge Ruiz-Cases}
\address{J. Ruiz-Cases. 
Departamento de Matem\'aticas, Universidad Aut\'onoma de Madrid, 28049 Madrid, Spain and Instituto de Ciencias Matemáticas ICMAT (CSIC-UAM-UCM-UC3M),
28049 Madrid, Spain}
\email{jorge.ruizc@uam.es}


\begin{document}
     
     \begin{abstract}
    In this paper we introduce two new fractional versions of the Laplacian.
The first one is based on the classical formula that writes the usual Laplacian
    as the sum of the eigenvalues of the Hessian. The second one
    comes from
    looking at the classical fractional Laplacian as the mean value
    (in the sphere) of the 1-dimensional fractional Laplacians
    in lines with directions in the sphere. To obtain this second new
    fractional operator we just replace the mean value by the mid-range
    of 1-dimensional fractional Laplacians
    with directions in the sphere.
    For these two new fractional operators we prove
    a comparison principle for viscosity sub and supersolutions 
    and then we obtain existence and uniqueness for the Dirichlet
    problem. Finally, we prove that for the first operator we recover the classical
    Laplacian in the limit as $s\nearrow 1$.
\end{abstract}

     \maketitle
     
       \section{Introduction}
       The classical Laplacian is the well known operator
       \begin{equation} \label{ec-Lapla-1}
       \Delta u (x) = \mbox{div} (\nabla u (x) ).
       \end{equation}
       This may be the best known and most famous second order differential operator.
       Written as in \eqref{ec-Lapla-1} it is an operator in divergence form. This allows to use techniques form calculus of variations a framework in which solutions are understood in a weak sense integrating against test functions (typically solutions are functions in the Sobolev space $H^1$). When one introduces coefficients in this context a natural operator to look at is
       \begin{equation} \label{ec-Lapla-A}
       L u (x) = \mbox{div} ( A (x) \nabla u (x) ),
       \end{equation}
       with a given matrix (that is usually assumed to be symmetric) with spatial dependence, $A (x)$, see for instance \cite{AM,No}. 
       
       A different way of writing the Laplacian is as
        \begin{equation} \label{ec-Lapla-2}
       \Delta u (x) = \mbox{tr} (D^2 u(x)) = \sum_{i=1}^N \lambda_i (D^2 u (x)). 
       \end{equation}
       Here $\lambda_1(D^2 u) \leq \lambda_2 (D^2 u) \leq ... \leq \lambda_N(D^2 u)$ stands for the eigenvalues of the Hessian, $D^2u = (\partial^2_{ij} u)_{ij}$. 
       This way of writing the Laplacian is not in divergence form but as an operator for which solutions are understood in viscosity sense \cite{CIL} (here solutions are just continuous functions and the operator is applied to smooth test functions that touches the solution from
       above or below).  
       Introducing coefficients thinking in this way one finds 
       \begin{equation} \label{ec-Lapla-A-2}
       F (D^2u) (x) = \mbox{tr} ( A (x) D^2 u (x) )
       \end{equation}
       with an $x-$dependent matrix $A (x)$, see \cite{CaCa}. 
       
       For the classical Laplacian both \eqref{ec-Lapla-1} and \eqref{ec-Lapla-2}
are equivalent ways of writing the same operator. For the Dirichlet problem for $\Delta u =0$ the notions of weak and viscosity solutions coincide (and in fact the Dirichlet problem
has a unique classical solution), see \cite{Hi} and \cite{RWZ} (the equivalence between weak and
viscosity solutions include quasi-linear equations, \cite{JLM}, and some non-local equations,
\cite{BM,dBdPO}). 
However, when one introduces coefficients, the operators \eqref{ec-Lapla-A} and 
\eqref{ec-Lapla-A-2} are not equivalent (in fact, the notion of weak solution using Sobolev spaces is not appropriate to deal with \eqref{ec-Lapla-A-2}). 

In recent years an operator that has become quite popular is the well-known
 fractional Laplacian defined as
         \begin{equation} \label{ec-Frac-Lapla-1}
      (- \Delta)^s u (x) = c(s) \, p.v.  \int_{\mathbb{R}^N} 
                             \frac{u(x)-u(y)}{|x-y|^{N+2s}} 
                             \, dy.
       \end{equation}
       Here $s \in (0,1)$, $p.v$ refers to the principal value of the integral and $c(s)$ is a constant that depends also on the dimension $N$
        and goes to zero as $(1-s)$ as $s \to 1$
       (we will make explicit and use the constant only in dimension 1). 
       For several different ways of writing the fractional Laplacian we refer to \cite{ten}.
       The operator in \eqref{ec-Frac-Lapla-1} is also well suited for variational methods 
(and typically solutions are functions in the fractional Sobolev space $H^s$). 
In this context one can introduce spatial dependence in the operator using a general symmetric kernel $k(x,y)$ (this symmetry assumption allows to integrate by parts and use variational techniques) and consider 
  \begin{equation} \label{ec-Frac-Lapla-1-A}
      (- \Delta)^s u (x) = p.v. \int_{\mathbb{R}^N} 
                             (u(x)-u(y) ) k(x,y) 
                             \, dy.
       \end{equation}
       Here a natural assumption is to ask that the kernel $k$ is comparable to the
       one of the fractional Laplacian, in the sense that, for two positive constants 
       $c_1$, $c_2,$ we have $c_1 |x-y|^{-N-2s} \leq k(x,y)
       \leq c_2 |x-y|^{-N-2s} $. 
       One possible choice of the kernel is given by $k(x,y) = |A (x-y)|^{-N-2s}$, obtaining an operator that is similar to \eqref{ec-Lapla-A}.
       When one wants to look for non-divergence form nonlocal operators one can consider
       non-symmetric kernels $k(x,y)$. These kind of operators have been intensively studied
       recently, we refer to \cite{CaS1,CaS2,CaS3,CLD,DiCKP,DiCKP2,Kim,Sil} and references therein.  
       
       However, up to now, there is no clear analogous to the classical way of understanding the Laplacian as the sum of the eigenvalues of the Hessian, as in \eqref{ec-Lapla-2}. 
       Our main goal in the paper is to introduce a new nonlocal operator that is a natural analogous to this way of looking at the classical Laplacian. To this end we first recall that, from the classical Courant-Hilbert formulas for the eigenvalues of a symmetric $N\times N$ matrix, we have
       \begin{equation} \label{eigen-matrix}
       \lambda_i (D^2 u(x) ) = \max_{{dim} (S) = N-i+1}
       \min_{\substack{z \in S, \\ |z| =1}} \langle D^2 u(x) z, z \rangle .
       \end{equation}
       Here the maximum is taken among all possible subspaces $S$ of $\mathbb{R}^N$ of dimension $N-i+1$ and the minimum among unitary vectors in $S$. 
For the Dirichlet problem for the equation $\lambda_i (D^2 u(x) ) =0$ we refer to 
\cite{BlancRossi}.
 Related operators are the truncated Laplacians studied in \cite{BGH11,BGH22,BGH33}.
Notice that \eqref{eigen-matrix} can be written as 
  \begin{equation} \label{eigen-matrix.88}
       \lambda_i (D^2 u(x) ) = \max_{{dim} (S) = N-i+1}
       \min_{ \substack{z \in S, \\ |z| =1}} \frac{\partial^2 u}{\partial z^2}(x) ,
       \end{equation}
       that is, the $i-$th eigenvalue is the max-min of the pure second derivatives
       computed 
       along directions in subspaces of dimension $N-i+1$.
       Therefore, a natural fractional version of the eigenvalues (that we will call fractional eigenvalues) 
       is given by 
       \begin{equation} \label{fract-eigen}
       \Lambda_i^s u (x) = \max_{{dim} (S) = N-i+1}
       \min_{ \substack{z \in S, \\ |z| =1}}  \left\{ c(s)\, p.v. \int_{\mathbb{R}} 
                             \frac{u(x+tz)-u(x)}{|t|^{1+2s}} 
                             \, dt
                         \right\},
       \end{equation}
       that is, we are computing the same max-min procedure as before, but now we are 
       taking the one-dimensional fractional derivative of order $2s$. 
       The Dirichlet problem for the first fractional eigenvalue is related to fractional convex envelopes, see \cite{dPQR}. Operators that are fractional analogous to truncated laplacians are studied in \cite{Bi1,Bi2,Bi3,dPQR}.
       Notice that due to their definition, the fractional eigenvalues are ordered
\[\Lambda_1^s u (x) \leq \Lambda_2^s u (x) \leq \cdots \leq \Lambda_{N-1}^s u (x) \leq \Lambda_N^s u (x).\]

       Now, let us introduce the operator that we call the {\it trace fractional Laplacian}, 
       \begin{equation} \label{tr-frac-Lapla}
       (-\Delta)^s_{tr} (u) (x) = -\sum_{i=1}^N \Lambda_i^s u (x) .      
        \end{equation}
       Here $\Lambda_i^s u$ is given by \eqref{fract-eigen}. 
             Notice that \eqref{tr-frac-Lapla} is not in divergence form and therefore we will use
       viscosity theory to study this operator. 
        We remark that this fractional version of the classical 
Laplacian is not equivalent to the usual fractional Laplacian given by 
\eqref{ec-Frac-Lapla-1}. This is a striking difference between the fractional setting and the classical local context (the variational fractional Laplacian does not coincide with the trace fractional Laplacian).

Our main goal in this paper is to show that the Dirichlet problem
for the trace fractional Laplacian is well posed in the framework of viscosity solutions.
Given a bounded domain $\Omega$ and an exterior datum
$g$ 
we will deal with
        \begin{equation} \label{Main-eq-intro} 
        \left\{
             \begin{array}{ll}
                \displaystyle (-\Delta)^s_{tr} (x) = - \sum_{i=1}^N \Lambda_i^s u (x) =0, \qquad & x \in \Omega, \\[6pt]
                         u (x) = g (x), \qquad & x \in \mathbb{R}^N\setminus \Omega.
             \end{array}
             \right.
         \end{equation}
         
         Since we want a continuous up to the boundary solution (some of our arguments
         requiere this property) we will assume 
         that 
         \begin{equation} \label{s>1/2}
         s \in (1/2,1).
         \end{equation}

Now, let us introduce a second fractional operator that we will
call the {\it mid-range fractional Laplacian}. 
To this end, notice that the classical fractional Laplacian, given by
\eqref{ec-Frac-Lapla-1}, of a smooth function
that decays at infinity can be written as
    \begin{equation} \label{ec-Frac-Lapla-1.radial}
    \begin{array}{l}
\displaystyle       (- \Delta)^s \phi (x) = c(s)\int_{\mathbb{R}^N} 
                             \frac{\phi (x)-\phi(y)}{|x-y|^{N+2s}} 
                             \, dy \\[10pt]
                           \displaystyle \qquad \qquad  
                           = -\vint_{\mathbb{S}^{N-1} }
                            c(s) \int_{\mathbb{R}} 
                             \frac{\phi (x+tz)-\phi(x)}{|t|^{1+2s}} 
                             \, dt \,  d\sigma (z),
                             \end{array}
       \end{equation}
       with $\mathbb{S}^{N-1} =\{ z \in \mathbb{R}^N : |z|=1 \}$.
       Therefore, the fractional Laplacian can be written (up to a negative constant) as the mean value
       (on the sphere) of the function $\Theta_\phi^s: \mathbb{S}^{N-1} \to \mathbb{R}$ given by
       $$
       \Theta_\phi^s (z) =
       \int_{\mathbb{R}} 
                             \frac{\phi (x+tz)-\phi(x)}{|t|^{1+2s}} 
                             \, dt.
       $$
       That is, to obtain the fractional Laplacian
       one computes the mean value in the directions of 
       the one-dimensional fractional derivatives of order $2s$. 
       With this idea in mind let us introduce 
       a different fractional operator. 
       Instead of the mean value we just take the mid-range
       (the measure of central tendency that is given by the average of the lowest and highest values in a set of data)
       of the function $\Theta_\phi^s$ in $\mathbb{S}^{N-1}$ and we
       obtain
       \begin{equation} \label{mid-range-Laplacian}
       \begin{array}{l}
       \displaystyle 
       (- \Delta)_{mid}^s \phi (x) = -
       \frac12 \min_{ |z| =1}  \Theta_\phi^s (z) - \frac12 \max_{ |z| =1}  \Theta_\phi^s (z)
       \\[10pt]
        \displaystyle \quad  = \! - \frac12
       \min_{ |z| =1} \!   \left\{ \!  c(s) \int_{\mathbb{R}} 
                             \frac{\phi(x+tz)-\phi(x)}{|t|^{1+2s}} 
                             \, dt \! 
                         \right\} \!  - \! \frac12 \max_{|z| =1} \!  \left\{\!  c(s) \int_{\mathbb{R}} 
                             \frac{\phi(x+tz)-\phi(x)}{|t|^{1+2s}} 
                             \, dt
                         \right\}\! 
                         \\[10pt]
        \displaystyle  \quad = -\frac12 \Lambda_1^s \phi (x) - \frac12
        \Lambda_N^s \phi (x).
                         \end{array}
          \end{equation}
          
          The mid-range and the trace fractional Laplacians are closely related.
          They are part of a
          large family of operators (those given in terms of combinations of
          fractional eigenvalues, we will add more comments on this in the final section of this paper). Moreover, in dimension two, that is, for $N=2$ in our notation, the mid-range and the trace fractional Laplacians coincide
    up to a constant
    $$(- \Delta)_{mid}^s u(x)  = - \frac12 \Lambda_1^s \phi (x) -  \frac12
        \Lambda_2^s \phi (x) = \frac12 (- \Delta)_{tr}^s u(x).$$
        In addition, remark that these two new fractional operators that we introduced here, 
          $(- \Delta)_{tr}^s$  and $(- \Delta)_{mid}^s$, are
          1-homogeneous (it holds that
          $(- \Delta)_{tr}^s (ku)= k (- \Delta)_{tr}^s (u)$ and 
          $(- \Delta)_{mid}^s (ku) = k (- \Delta)_{mid}^s (u) $ for $k\in \mathbb{R}$),
          they have a scaling invariance of order $2s$ (for $w(x) = u(kx)$ we have
          $(- \Delta)_{tr}^s (w)(x)) = k^{2s} (- \Delta)_{tr}^s (u)(kx) $
          and $(- \Delta)_{mid}^s (w )(x) = k^{2s} (- \Delta)_{mid}^s (u)(kx) $),
          and are invariant
          under rotations, as the usual Laplacians (both local and fractional) are.

        For the mid-range fractional Laplacian we also study the Dirichlet problem,
        that in this case reads as
          \begin{equation} \label{Main-eq-section.iiii} 
        \left\{
             \begin{array}{ll}
                \displaystyle \Lambda_1^s u (x)+ \Lambda_N^s u (x)  =0, \qquad & x \in \Omega, \\[6pt]
                         u (x) = g (x), \qquad & x \in \mathbb{R}^N\setminus \Omega.
             \end{array}
             \right.
         \end{equation}   
Here we dropped the constant $1/2$ in front of the fractional eigenvalues to simplify the notation.

         Our first result says that when the domain is smooth and
         the exterior datum $g$ is continuous and bounded, 
         there is a unique viscosity solution for
         \eqref{Main-eq-intro} or for \eqref{Main-eq-section.iiii}.
         
          \begin{theorem} \label{teo.existencia-uni.intro} Fix $s\in (1/2,1)$.
        Given a $C^2$ bounded domain $\Omega$, and a continuous and bounded exterior datum $g$, there exists a unique viscosity solution to \eqref{Main-eq-intro}
        or to \eqref{Main-eq-section.iiii}.
      \end{theorem}

         These operators share many properties
         with  the corresponding local Laplacian and the classical fractional Laplacian  
         (as the validity of a comparison principle
         and a strong maximum principle and the continuous
         dependence of the solutions on the exterior data).

         \begin{theorem} \label{teo.compar.intro}
         Under the hypothesis of Theorem \ref{teo.existencia-uni.intro},
         a comparison principle holds,
        let $u_1$ and $u_2$ denote the solutions to \eqref{Main-eq-intro}
        (or to \eqref{Main-eq-section.iiii})
        with exterior data $g_1$ and $g_2$ respectively, then
        \begin{equation} \label{compar}
        g_1 \geq g_2 \mbox{ in } \mathbb{R}^N \setminus \Omega
        \implies 
u_1 \geq  u_2  \mbox{ in }\Omega  
                \end{equation}
        and, as a consequence, the solution depends continuously
        on the exterior data; it holds that
        \begin{equation}
        \| u_1 - u_2 \|_{L^\infty (\Omega)} \leq \| g_1 - g_2 \|_{L^\infty (\mathbb{R}^N \setminus \Omega)}.
        \end{equation}
        
        Moreover, a strong maximum principle holds
        both for \eqref{Main-eq-intro}
        and for \eqref{Main-eq-section.iiii}. If there exists $x_0 \in \Omega$
        such that
        $$
       \sup_{y\in \mathbb{R}^N \setminus \Omega} g(y) \leq u(x_0)
        $$
        then, the solution and the exterior datum are constant, $u\equiv g \equiv cte$. 
         \end{theorem}
         
         A striking difference with the classical Laplacian (and also with the 
         classical fractional Laplacian) is that these operators are nonlinear.

          \begin{theorem} \label{teo-no-lineal.intro}
         The problems \eqref{Main-eq-intro} and \eqref{Main-eq-section.iiii} are nonlinear problems. There exists
         data $g_1$, $g_2$ such that
         \begin{equation} \label{compar.4455}
        g_1\gneqq  g_2 \mbox{ in } \mathbb{R}^N \setminus \Omega,
        \end{equation}
       but the corresponding solutions verify 
       \begin{equation} \label{real-madrid} 
u_1 \equiv  u_2  \mbox{ in }\Omega. 
                \end{equation}
                              \end{theorem}
                 
                 That the operators are nonlinear is due to the fact that there are maxima and minima
         involved in the definition of $\Delta_{tr}^s$ and $\Delta_{mid}^s$. In fact, what is really surprising is that
         the usual local Laplacian, written as 
         $$
         \Delta u (x) = \sum_{i=1}^N \lambda_i (D^2 u(x))
         = \sum_{i=1}^N \max_{{dim} (S) = N-i+1}
       \min_{ \substack{z \in S, \\ |z| =1}} \frac{\partial^2 u}{\partial z^2}(x),
         $$
         is a linear operator.

         In addition, we study the limit as $s\nearrow 1$ and prove that solutions to 
         \eqref{Main-eq-intro} converge uniformly to the unique solution to the
         Dirichlet problem for the classical local Laplacian 
          \begin{equation} \label{Dirich-local-Lapla} 
        \left\{
             \begin{array}{ll}
                \displaystyle \Delta u(x)  = 0 \qquad & x \in \Omega, \\[6pt]
                         u (x) = g (x) \qquad & x \in \partial \Omega.
             \end{array}
             \right.
         \end{equation}

         Here we will use the explicit constant 
         \begin{equation} \label{cte}
         c(s) = 
         \frac{2^{2s}s\Gamma  (s+1/2)}{\pi^{1/2} \Gamma  (1-s)}
         \end{equation}
          that 
         appears front of the integral in \eqref{fract-eigen}. 
         We just remark that any $c(s)$ such that $c(s) \sim (1-s)$ will also work
         when taking this limit, but the
         explicit form of \eqref{cte} is the one that corresponds to the 1-dimensional fractional Laplacian.

          \begin{theorem} \label{teo.limite-s-a-1.intro}
         Let $\Omega$ be a $C^2$ bounded domain and fix $g\in C (\mathbb{R}^N \setminus \Omega)$ and bounded. Let $u_s$ denote the unique solution to
         to the problem for
         the trace fractional Laplacian, \eqref{Main-eq-intro}.
         Then, it holds that
         $$
         \lim_{s \nearrow 1} u_s = u
         $$
         in $C (\overline{\Omega})$. The limit $u$ is given by the unique solution to
         \eqref{Dirich-local-Lapla}.
         
         Moreover, when $u_s$ is the unique solution to the problem for
         the mid-range fractional Laplacian, \eqref{Main-eq-section.iiii},
         it holds that
         $$
         \lim_{s \nearrow 1} u_s = u
         $$
         in $C (\overline{\Omega})$ where the limit $u$ is given by the unique solution to
         the local problem
          \begin{equation} \label{Dirich-local-Lapla.N=2} 
        \left\{
             \begin{array}{ll}
                \displaystyle \lambda_1 (D^2 u)(x) + \lambda_N (D^2 u) (x) = 0 \qquad & x \in \Omega, \\[6pt]
                         u (x) = g (x) \qquad & x \in \partial \Omega.
             \end{array}
             \right.
         \end{equation}
         \end{theorem}
         
         Remark that the first limit result also holds for the usual fractional Laplacian
         \eqref{ec-Frac-Lapla-1} with the appropriate $N$ dimensional constant $c(s)$.

         In these limits, as $s\nearrow 1$, the limit is unique and is characterized as
         the solution to \eqref{Dirich-local-Lapla} or to \eqref{Dirich-local-Lapla.N=2}. Hence, we have convergence of the whole family $u_s$ as
         $s \nearrow 1$ (not only along subsequences).
         
               Concerning regularity of solutions we quote the recent paper \cite{HolderRef} where the authors prove a
         Hölder regularity result for solutions to $\Lambda_1^s u (x)+ \Lambda_N^s u (x)  =f(x)$ in $\Omega$ with homogeneous Dirichlet boundary conditions, 
 $u (x) = 0$ in $\mathbb{R}^N\setminus \Omega$ and $s$ close to 1.
 Regularity issues for these operators is a delicate issue since they are very degenerate.

        Now, to end the introduction, let us comment briefly on
        the ideas and methods used in the proofs.  
         
         In most of our arguments, to simplify the notation and clarify
         the ideas used in the proofs we will only include the details for the mid-range fractional Laplacian, 
$ (-\Delta)^s_{mid} (u) (x) $, that involve the smallest and the largest fractional eigenvalues.
         Since for this operator the main difficulties arise, we will just briefly comment
         on how to extend the results for the fractional trace Laplacian, $ (-\Delta)^s_{tr} (u) (x) $, in which
         the intermediate eigenvalues appear. 
         Recall again that for $N=2$ the mid-range fractional Laplacian and the trace fractional
Laplacian coincide.
         In addition, we will drop the notation $p.v.$ in front of the
         integrals (that need to be understood in the principal value sense when
         appropriate) and, when we prove results for
         a fixed $s$, we will also drop the constant $c(s)$.
            
            Since our problem is fully nonlinear we use the concept 
         of viscosity solutions in a nonlocal framework.
         We will use ideas from \cite{1,BarImb,BarChasImb}.
         First, we prove a comparison principle for sub and supersolutions to
         our problems 
          \eqref{Main-eq-intro} and \eqref{Main-eq-section.iiii}. The proof follows ideas from \cite{BarChasImb}
          and \cite{dPQR}. Notice that in \cite{dPQR} it is assumed that the domain is
          strictly convex. This condition is not needed here since  
          the maximum and the minimum among directions are
          involved in our operator and we can choose any direction 
          either in the max or in the min
          to obtain a super or a subsolution. Once we have a comparison principle
          existence and uniqueness of solutions are an easy consequence of
          Perron's method. 
          
          To show that the operator is nonlinear we find 
          a domain in two dimensions, a smooth function $u$ inside $\Omega$ and
          an exterior datum $g$ such that 
          the lines corresponding to directions associated with the maximum and the infimum
          of the $1-$dimensional fractional derivatives of $u$ inside $\Omega$
          do not intersect a region outside $\Omega$. Then, we show that the exterior datum $g$ 
          can be slightly perturbed in that region keeping the same $u$ as the solution inside $\Omega$. In this way we find two different exterior data (that, in addition, are ordered) with the same solution inside the domain and we conclude the nonlinearity of our problem from the strong maximum principle.

          To recover the usual Laplacian in the limit as $s\nearrow 1$
          we just have to observe that with the choice of the constant $c(s)$ given in
          \eqref{cte} it holds that the 1-dimensional fractional Laplacian converges
          to the usual second derivative as $s\nearrow 1$ (for $c(s) \sim (1-s)$
          we just obtain a multiple of the Laplacian in the limit). Then, the proof 
          follows just taking care of the $\max$ and $\min$ involved in the
          fractional eigenvalues using viscosity tricks. The uniform convergence
         in Theorem \ref{teo.limite-s-a-1.intro}
 will follow from the fact that we show that the upper  and lower half-relaxed 
         limits of $\{u_s\}_s$ coincide.

          Finally, let us point out again that to compute the 1-dimensional fractional Laplacian
          in directions $z$ on some particular test functions we need to restrict ourselves
          to consider only the case $s>1/2$.
         We believe that without this condition solutions may noy be continuous 
         up to the boundary of the domain.

         The paper is organized as follows: In Section \ref{sect-compar}
         we prove the comparison principle for viscosity sub and supersolutions to our
         problems and then we obtain existence and uniqueness for the Dirichlet
         problems.
         In Section \ref{sect-nonlinear} we show that the operators
         are nonlinear. 
         In Section \ref{sect-lim-s} we deal with the limit as $s \nearrow 1$.
         Finally, in Section \ref{sect-extension}
       we will comment on possible extensions of our results and describe how to
       introduce coefficients in fractional trace operators.

     \section{Comparison principle. Existence and uniqueness of solutions} \label{sect-compar}

         The main result in this section is to prove a comparison principle for the problem 
 \begin{equation} \label{Main-eq-section.44} 
        \left\{
             \begin{array}{ll}
                \displaystyle \sum_{i=1}^N \Lambda_i^s u (x) =0, \qquad & x \in \Omega, \\[6pt]
                         u (x) = g (x), \qquad & x \in \mathbb{R}^N\setminus \Omega.
             \end{array}
             \right.
         \end{equation}         
         To this end, we borrow ideas from \cite{BarChasImb} (see also \cite{dPQR}).
         
         Since in this section $s$ is fixed we drop the constant $c(s)$ that plays no role
         in our arguments. 
         Also, as mentioned in the introduction, to simplify the notation in the proofs we will only analyze in detail
         the problem
          \begin{equation} \label{Main-eq-section} 
        \left\{
             \begin{array}{ll}
                \displaystyle \Lambda_1^s u (x)+ \Lambda_N^s u (x)  =0, \qquad & x \in \Omega, \\[6pt]
                         u (x) = g (x), \qquad & x \in \mathbb{R}^N\setminus \Omega,
             \end{array}
             \right.
         \end{equation}   
     and at the end of the section comment on how to obtain the results for
     \eqref{Main-eq-section.44}. In fact, we can consider any combination with
     nonnegative
     coefficients of fractional
     eigenvalues as long as $\Lambda_1^s u$ and $\Lambda_N^s u$ appear.

     \subsection{Basic notations and definition of solution.} 
         We use the notion of viscosity solution from \cite{BarChasImb}, which is the nonlocal extension of 
         the classical theory, see \cite{CIL}.

         To state the precise notion of solution, we need the following: 
         Given $g\colon\R^N\setminus\Omega \to \R,$  for a function $u \colon \overline{\Omega} \to \R$,
         we define the upper $g$-extension of $u$ as
         \begin{equation*}
             u^g(x) \coloneqq 
                 \left \{ 
                     \begin{array}{ll} 
                         u(x) \quad &  \ x \in \Omega, \\[6pt] 
                         g(x) \quad &  \ x \in \R^N\setminus\overline{\Omega},\\[6pt]
                         \max \{ u(x), g(x) \} \quad &  \ x \in \partial \Omega.
                     \end{array} 
                 \right . 
         \end{equation*}
         In the analogous way we define $u_g$, the lower $g$-extension of $u$, replacing $\max$ by $\min$.

          Now we introduce the definition of the upper and lower semicontinuous envelope,
          that we will denote by $\tilde{u}$ and $\undertilde{u}$ respectively of $u$, 
          that are given by
         \[
                 \tilde{u}(x)\coloneqq\inf_{r>0}\sup \Big\{u(y)\colon y\in B(y,r) \Big\}
                 \]
                 and
                 \[
                 \undertilde{u}(x)=\sup_{r>0}\inf \Big\{u(y)\colon y\in B(y,r)\Big\}.
         \]  
     
         An important fact, that can be easily 
         verified, is that for 
         any continuous function $g \colon \R^N\setminus\Omega \to \R$ and
         any upper semicontinuous function $u \colon \overline{\Omega} \to \R,$  it holds that
         \[
             u^g = \tilde w,\quad \mbox{ with} \quad w=u \mathbf{1}_{\overline{\Omega}} + g \mathbf{1}_{\R^N\setminus\overline{\Omega}}
            \text{ in }\R^N.
         \] 
         Here $\mathbf{1}_{A}$
         denotes the indicator function of a set $A$ in $\mathbb{R}^N.$

         We now introduce a useful notation, for $\delta > 0$ we write
             \[	
             E_{z, \delta}(u, \phi, x)\coloneqq I^1_{z, \delta}(\phi, x)+I^2_{z, \delta}(u, x)
         \] 
         with
         \[
             \begin{array}{l}
                 \displaystyle 
                     I^1_{z, \delta}(\phi, x)
                     \coloneqq\int_{-\delta}^\delta \frac{\phi (x+t z)-\phi(x)}{|t|^{1+2s}}dt,\\[15pt]
                 \displaystyle  
                     I^2_{z, \delta}(u, x)\coloneqq\int_{\RR \setminus (-\delta,\delta)} 
                         \frac{u^g( x+t z)-u( x)}{|t|^{1+2s}}dt,
             \end{array}
         \]
         and then define 
         \begin{equation*}
         \begin{array}{l}
         \displaystyle
             E_\delta(u, \phi, x) \coloneqq 
             - \left[ \inf_{ |z| =1} 
                 E_{z, \delta}(u, \phi, x)
                + 
                   \sup_{ |z| =1} 
                 E_{z, \delta}(u, \phi, x)
                  \right].
                 \end{array}
         \end{equation*}

         Now, with these notations at hand, we can introduce our notion of viscosity solution
         to \eqref{Main-eq-section} testing with $N-$dimensional functions as usual. 
     
         \begin{definition}\label{defsol.44}
             A bounded upper semicontinuous function $u \colon \R^N \to \R$  
             is a viscosity subsolution to the Dirichlet problem \eqref{Main-eq-section} if
             $u \leq g$ in $\R^N\setminus\overline{\Omega}$ and if 
             for each $\delta > 0$ and $\phi \in C^2(\R^N)$ such that $x_0$ is a maximum 
             point of $u - \phi$ in $B_\delta(x_0)$, then
             \begin{equation*}
                 \begin{array}{ll}
                     \displaystyle
                         E_\delta(u^g, \phi, x_0) \leq 0 & \quad \mbox{if} \ x_0 \in \Omega, \\[6pt]
                     \displaystyle 
                         \min \left\{ E_\delta(u^g, \phi, x_0), u(x_0) - g(x_0) \right\} \leq  0 
                         & \quad \mbox{if} \ x_0 \in \partial \Omega.
                 \end{array}
             \end{equation*}
     
             In an analogous way, we define viscosity supersolutions (reversing the inequalities
             and replacing $u^g$ by $u_g$)
             and viscosity solutions (asking that $u$ is both a supersolution and a subsolution) 
             to \eqref{Main-eq-section}.
         \end{definition}
         
         \begin{remark}
         For the definition of a viscosity solution to
         the Dirichlet problem for the trace fractional Laplacian, \eqref{Main-eq-section.44},
         we just have to take
       $$
         \begin{array}{l}
         \displaystyle
             E_\delta(u, \phi, x) \coloneqq 
             - \left[ \sum_{i=1}^N \sup_{{dim} (S) = N-i+1}
       \inf_{ \substack{z \in S, \\ |z| =1}}  
                 E_{z, \delta}(u, \phi, x)
                  \right]
                 \end{array}
       $$
       in the previous definition. 
         \end{remark}
     
     \subsection{Comparison principle.}
            Now, our goal is to prove a comparison principle
            between sub and supersolutions to \eqref{Main-eq-section}.
          
     In order to prove the comparison principle, we first show that sub and supersolutions
     behave well on the boundary of the domain.

         \begin{theorem}\label{condicion-borde.77} 
            Assume that  $g\in C(\R^N\setminus\Omega)$ is bounded
            and that $\Omega$ is a bounded $C^2-$domain. 
            Let $u,v\colon \mathbb{R}^N\to\mathbb{R}$ 			
            be  viscosity sub and supersolution of \eqref{Main-eq-section} in $\Omega$, 
            in the sense of Definition \ref{defsol.44}, respectively. Then,
            \begin{enumerate}[(i)]
                \item $u\le g$ on $\partial \Omega$; 
                
                \medskip
                
                \item $v\ge g$ on $\partial \Omega$.
            \end{enumerate}
        \end{theorem}
    
        \begin{proof}		
            We begin by proving {\it (i)}. Suppose by contradiction that 
            there is $x_0\in\partial\Omega$ such that 
            \[
                \mu \colon=u(x_0)-g(x_0)>0.
            \] 
            Hence, we have that $u^g(x_0)=u(x_0)$.
            Since $g$ is continuous, there exists $R_0 > 0$ such that 
            \begin{equation}\label{eq:attainablility1}
                |g(x_0)-g(y)| \leq \frac{\mu}4
                \qquad \forall y\in B(x_0,2R_0)\cap (\RR^N\setminus\Omega).
            \end{equation}
             We may with  no loss of
             generality assume that $R_0<\max\{\|x-y\|\colon x,y\in\overline{\Omega}\}.$
             
            We now introduce two auxiliary functions:
            \begin{itemize}
                \item $a\colon\RR^N\to \RR$, a smooth bounded function such that 
                    $a(0)=0,$ $a(y)>0$ if $y\neq0,$ 
                    \[
                        \liminf\limits_{|y|\to\infty }a(y)>0
                    \] 
                    and $D^2 a$ is bounded; for instance we can just take
                    $$
                    a(y) = |y|^2.
                    $$
                \item $b\colon\RR\to \RR$, a smooth bounded and increasing function 
                    which is concave in $(0,+\infty),$ and \\
                    such that $b(0)=0,$ $$b(t)>-\frac{\mu}4$$ in $\RR,$  
                    $b^{\prime}(0)=k_1$ and $b^{\prime\prime}(0)=-k_2$ for two
                    positive constants $k_1>0$, $k_2>0.$ 
            \end{itemize}

            Next, we use these two functions to define for any $\varepsilon>0$ the penalized test function
            \[
                \omega_\varepsilon (y)\coloneqq\frac{a(y-x_0)}{\varepsilon}+
                b\left(\frac{\dist(y)}{\varepsilon}\right),
            \]
            here $\dist$ is a smooth extension of the signed distance to the boundary, $\partial \Omega$. It
            is at this point where we use the $C^2$ regularity of $\partial \Omega$ to ensure
            that $\dist$ is $C^2$. 
            
            Now, we introduce
            \[
                \Psi_\varepsilon(y)= u^g(y)-\omega_\varepsilon(y)
            \]
            that is upper semicontinuous for any $\varepsilon$ small. 
            Then, for any  $\varepsilon$ small, $\Psi_\varepsilon$ attains a global maximum at a point
            $x_\varepsilon.$
            Therefore, we have 
            \begin{equation}\label{eq:attainablility2}
                u^g(x_\varepsilon)-\omega_\varepsilon(x_\varepsilon)\geq 
                u^g(x_0)-\omega_\varepsilon(x_0)=u^g(x_0)
            \end{equation}
            and hence,
            \[
                 \frac{a(x_\varepsilon-x_0)}{\varepsilon}\leq 
                     u^g(x_\varepsilon)-u(x_0)-b\left(\frac{\dist(x_\varepsilon)}{\varepsilon}\right)
                 \leq
                 u^g(x_\varepsilon)-u(x_0)+\|b\|_\infty.
            \]
            From here, we get that 
            \begin{equation}\label{eq:attainablility3}
                x_\varepsilon\to x_0 \qquad \text{ as } \varepsilon \to 0,
            \end{equation}
            In particular, $x_\varepsilon\in B(x_0,2R_0)$ for any $\varepsilon$ small enough. 
            Now, using again the properties of $a$ and $b$ we get
            \[
                g(x_0)+\mu=u(x_0)\leq u^g(x_{\varepsilon})- 
                \omega(x_{\varepsilon})
                \leq u^g(x_{\varepsilon})+\frac{\mu}{4}.
            \]
            Therefore $x_\varepsilon\in\overline{\Omega}.$ 
              Notice that we used the function $a$ to penalize that
            $x_\varepsilon$ is far from $x_0$ and the function $b$ with
            the signed distance to $\partial \Omega$
            to obtain that $x_\varepsilon$ is not in $\mathbb{R}^N \setminus \overline{\Omega}$.
        
            Since $a$ is non-negative, by \eqref{eq:attainablility2}, we have
            \[
                b\left(\frac{\dist(x_{\varepsilon} )}{\varepsilon}\right)
                \le u^g(x_{\varepsilon})-u(x_{0}).
            \]
            Hence, due to the fact that $u$ is upper semicontinuous, we obtain
            \[
                b\left(\frac{\dist(x_{\varepsilon} )}{\varepsilon}\right)
                \to 0 \qquad \mbox{ as } \varepsilon\to 0.
            \]
            Thus, we have that
            \begin{equation}\label{eq:attainablility4}
                \frac{a\left(x-x_0\right)}{\varepsilon}
                \to 0,\quad
                \frac{\dist(x_{\varepsilon} )}{\varepsilon}
                \to 0 \quad \mbox{ and } \quad u^g(x_{\varepsilon})\to u(x_0)
            \end{equation}
            as $\varepsilon\to0.$

            Since $u(x_\varepsilon)\to u(x_0)=g(x_0)+\mu$ and $g$ is continuous,
            if $x_\varepsilon\in\partial\Omega$ then $u(x_\varepsilon)>g(x_\varepsilon)$ 
            for any $\varepsilon$ small enough. Now, using that $u$ is a viscosity subsolution  
            of \eqref{Main-eq-section} 
in $\Omega$ in the sense of Definition \ref{defsol.44}, we have that
            \begin{equation}\label{eq:attainablility5}
                E_\delta(u^g, \omega_\varepsilon, x_\varepsilon)\leq 0.
            \end{equation}
            
            \textit{Case 1:} $x_\varepsilon\in \Omega.$ 
            
            Here we need to introduce changes
            in the arguments used in \cite{BdpQR}. The key point 
            is that the directions $z$ that are associated to the $\max$ and $\min$
            that appear in the equation
                      \begin{equation} \label{Main-eq-section.888} 
                      \begin{array}{l}
                      \displaystyle 
         \Lambda_1^s u (x)+ \Lambda_N^s u (x) \\[10pt]
         \displaystyle \quad = 
       \min_{|z| =1}  \left\{ \int_{\mathbb{R}} 
                             \frac{u(x+tz)-u(x)}{|t|^{1+2s}} 
                             \, dt
                         \right\} +
                         \max_{|z| =1}  \left\{ \int_{\mathbb{R}} 
                             \frac{u(x+tz)-u(x)}{|t|^{1+2s}} 
                             \, dt
                         \right\} =0,  
                         \end{array}       
         \end{equation}   
         may be different (and we will take advantage of this fact).
         
         We have 
          \begin{equation} \label{sum}
         \begin{array}{l}
         \displaystyle
             E_\delta(u, \phi, x_\varepsilon) \coloneqq 
             - \left[ \inf_{|z| =1}
                 E_{z, \delta}(u, \omega_\varepsilon, x_\varepsilon)                
                + 
                   \sup_{|z| =1}
                 E_{z, \delta}(u, \omega_\varepsilon, x_\varepsilon)
                 \right] \leq 0.
                 \end{array}
         \end{equation}

            To bound the infimum we choose 
            $$z_\varepsilon = \frac{\bar{x}_\varepsilon-x_\varepsilon}{|\bar{x}_\varepsilon-x_\varepsilon|},$$ 
            with $\bar{x}_\varepsilon$ the point on $\partial \Omega$ such that $\delta_\varepsilon = d(x_\varepsilon) = |\bar{x}_\varepsilon - x_\varepsilon|$. We can assume that $\delta_\varepsilon < \varepsilon$, since 
            we have 
            \begin{equation}
            \label{borraron}
            \frac{\delta_\varepsilon}{\varepsilon} \rightarrow 0.
            \end{equation}
            With this choice,             
            \[
            \begin{array}{l}
            \displaystyle 
            \inf_{|z|=1} E_{z, \delta_\varepsilon}(u, \omega_\varepsilon, x_\varepsilon)  \\[10pt]
            \displaystyle  \qquad \leq \frac{C}{\delta_{\varepsilon}^{2 s}}\left[\left(\frac{\delta_{\varepsilon}}{\varepsilon}\right)^2+(1+K \mu)\left(\frac{\delta_{\varepsilon}}{\varepsilon}\right)^{2 s}+\kappa(\varepsilon, s) \delta_{\varepsilon}^{2 s}-K \mu\right] 
            \\[10pt]
            \displaystyle  \qquad
            \sim -\frac{C}{\delta_{\varepsilon}^{2 s}} K\mu = -\frac{C_1}{\delta_{\varepsilon}^{2 s}}.
            \end{array}
            \]

            Now, for the supremum we argue as follows: the main idea is that
            what we obtained with our choice in the infimum is negative
            enough to absorb the supremum and reach a contradiction at the end
            (the contradiction arrives from the fact that the sum 
            of the infimum and the supremum is greater 
            or equal to zero, see \eqref{sum},  
            but the infimum is negative enough in order to obtain that the sum
            is also negative).   
            
            Given $\eta_\varepsilon>0$, we choose a
         direction, $\widetilde{z}_\varepsilon$, such that 
            \[E_{z_\varepsilon, \delta_\varepsilon}(u, \omega_\varepsilon, x_\varepsilon) + \eta_\varepsilon \geq \sup_{|z|=1} E_{z, \delta_\varepsilon}(u, \omega_\varepsilon, x_\varepsilon). \]
        Then, 
            \[
            \begin{array}{l}
            \displaystyle
            -\eta_\varepsilon + \sup_{|z|=1} E_{z, \delta_\varepsilon}(u, \omega_\varepsilon, x_\varepsilon) 
             \\[10pt]
            \qquad \displaystyle 
           \leq E_{z_\varepsilon, \delta_\varepsilon}(u, \omega_\varepsilon, x_\varepsilon) 
            \\[10pt]
            \qquad \displaystyle \leq E_{z_\varepsilon, \varepsilon}(u, \omega_\varepsilon, x_\varepsilon) 
            = I^1_{z, \varepsilon}(\omega_\varepsilon, x_\varepsilon) + I^2_{z, \varepsilon}(u, x_\varepsilon).
            \end{array} \]
    We can prove both integrals are of order $\varepsilon^{-2s}$ in a similar way as the integrals in the infimum: we estimate $I^1_{z, \varepsilon}(\omega_\varepsilon, x_\varepsilon)$ using the Taylor expansion of $\omega_\varepsilon$ and we estimate $I^2_{z, \varepsilon}(u, x_\varepsilon)$ simply by using the boundedness of $u^g$. Thus, 
            \[
            \displaystyle 
            -\eta_\varepsilon + \sup_{|z|=1} E_{z, \delta_\varepsilon}(u, \omega_\varepsilon, x_\varepsilon)  \leq - \frac{C_2}{\varepsilon^{2s}}
            \]
            Adding the bounds for the infimum and the supremum and using \eqref{borraron} we obtain 
            \[-\eta_\varepsilon \lesssim \frac{1}{\delta_\varepsilon^{2s}} \left[-C_1 +  C_2 \left(\frac{\delta_\varepsilon}{\varepsilon}\right)^{2s} \right] \sim -\frac{C_1}{\delta_\varepsilon^{2s}}\]
            and we reach a contradiction taking $\eta_\varepsilon$ small enough.
            
            
            \textit{Case 2:} $x_\varepsilon\in\partial\Omega.$ 
                    
            Since we are dealing with subsolutions, 
            in the associated inequality we can choose any direction
            to obtain a bound for the infimum, but we have to take
            care of the supremum. When one deals with
            supersolutions the situation is exactly the opposite. Notice that in our problem 
            \eqref{Main-eq-section}
            both the infimum and the supremum appear.
            
            Hence, as we are dealing with subsolutions, for the
            supremum, given $\eta_\varepsilon $, we have a direction $\widetilde{z}_\varepsilon$ as before. On the other hand, for the infimum we are free to choose the direction. We choose
             $z_\varepsilon$ that points inwards the domain
             (for example the inner unit normal to $\partial \Omega$ at 
             the point $x_\varepsilon$ will do the job). We also choose
             a distance $\delta_\varepsilon$ such that 
             $\frac{\delta_\varepsilon}{\varepsilon} \rightarrow 0$. 
            
           Now, we use exactly the same computations as in the previous case.
           For the infimum we use $\delta_\varepsilon$ and $z_\varepsilon$ and for the supremum $\varepsilon$ and $\widetilde{z}_\varepsilon$. As in the previous case, we reach a contradiction
           since the infimum is negative enough and we have a control of the supremum 
           in such a way that the sum is still negative. 
            
            In the case of supersolutions 
            the arguments works in an analogous way since we can interchange
            the roles of the infimum and the supremum in the previous 
            computations (notice that for supersolutions we are reversing the inequalities). 
                        \end{proof}

     \begin{remark} To deal with the problem involving 
     the trace fractional Laplacian $\Delta_{tr}^s$, we just observe that, for the
     terms that involve intermediate fractional eigenvalues,
      \begin{equation} \label{fract-eigen.99}
       \Lambda_i^s u (x) = \max_{{dim} (S) = N-i+1}
       \min_{ \substack{z \in S, \\ |z| =1}}  \left\{  \int_{\mathbb{R}} 
                             \frac{u(x+tz)-u(x)}{|t|^{1+2s}} 
                             \, dt
                         \right\},
       \end{equation}
       with $i=2,...,N-1$, we can choose a direction that almost reaches this quantity. 
       In fact, when we deal with subsolutions we use that, since $\Lambda_1^s u$ involves only the 
       infimum, we have freedom to choose the direction. For the terms that involve the  other eigenvalues we have that,
     given $\eta_\varepsilon>0$, we can choose a
         direction, $\widetilde{z}_\varepsilon$, such that 
            \[E_{z_\varepsilon, \delta_\varepsilon}(u, \omega_\varepsilon, x_\varepsilon) + \eta_\varepsilon \geq \sup_{{dim} (S) = N-i+1}
       \inf_{ \substack{z \in S, \\ |z| =1}}  
       E_{z, \delta_\varepsilon}(u, \omega_\varepsilon, x_\varepsilon) \]
       and then, by the same computations that we made before
       (we just have to add a finite number of
        terms involving $ \eta_\varepsilon$) we reach a contradiction.
       
     Notice that, when dealing with supersolutions we use the freedom in the choice
     of the direction in the term that comes from $\Lambda_N^s u$ (this
     is the term that involves a supremum) and bound all the other terms
     (that involve $\inf$/$\sup$). 
       \end{remark}
            
           \begin{remark}
           In \cite{BdpQR} is is used that the domain is strictly convex.
           Here we do not need this condition, since we have both the infimum
           and the supremum among directions in our operator and hence we can choose
           the direction in one of the terms (the one with the infimum or the one with the supremum)
           when dealing with sub and super solutions. 
           \end{remark}

            Now, we are ready to state and prove the comparison principle for
     sub and super solutions to our problem.
     In this proof we again follow ideas from \cite{BdpQR} but we have to 
     introduce a different function $\Psi_\varepsilon$ (see below)
     and the integrals that appear are also split in a different way.
     Again, here we are not assuming that the domain is strictly convex
     (as was needed for the arguments in \cite{BdpQR}).
            
            \begin{theorem}\label{teo:cp}
                Assume that $g\in C(\R^N\setminus\Omega)$ is
                bounded and that $\Omega$ is a bounded $C^2-$domain. 
                Let $u,v\colon \mathbb{R}^N\to\mathbb{R}$ 			
                be a viscosity sub and supersolution to \eqref{Main-eq-section} 
                in $\Omega$, in the sense of Definition \ref{defsol.44}, then $$u\le v$$ in $\mathbb{R}^N$.
            \end{theorem}
            \begin{proof}
                Let $R>0$ be big enough such that $\overline{\Omega} \subset \overline{B_R}$. Let
                \[
                    M\coloneqq\sup \Big\{ u^g(x)-v_g(x)\colon x\in \overline{B_R} \Big\}.
                \]
                As usual, we argue by contradiction, that is, we assume that $M>0.$
                Since $u^g$ and $v_g$ are upper and lower semicontinuous functions, 
                \[
                    S\coloneqq \sup\Big\{u^g(x)\colon x\in\overline{B_R} \Big\}-\inf 
                    \Big\{v_g(x)\colon x\in\overline{B_R} \Big\}<\infty.
                \]
                
                For any $\varepsilon>0,$ we define
                \[
                    \Psi_\varepsilon(x,y)\coloneqq u^g(x)-v_g(y)-\dfrac{|x-y|^2}\varepsilon.
                \] 
                Observe that
                \begin{equation}\label{eq:cp1}
                    M\le M_\varepsilon\coloneqq\sup \Big\{\Psi_\varepsilon(x,y)\colon (x,y)\in\overline{B_R}\times\overline{B_R}\Big\}
                    \le S.
                \end{equation}
                Moreover, $M_{\varepsilon_1}\le M_{\varepsilon_2}$ for all $\varepsilon_1\le\varepsilon_2.$ Then, there
                exists the limit
                \begin{equation}\label{eq:cp2}
                    \lim_{\varepsilon \to 0^+} M_\varepsilon=\overline{M} >0 .
                \end{equation}
                
                On the 	other hand, since $u^g$ and $-v_g$ are upper semicontinuous functions, for any $\varepsilon$,
                $\Psi_\varepsilon$ is an upper semicontinuous function. Thus, there is 
                $(x_\varepsilon,y_\varepsilon)\in\overline{B_R}\times\overline{B_R}$ such that
                \begin{equation}\label{eq:xeye}
                    M_\varepsilon=\Psi_\varepsilon(x_\varepsilon,y_\varepsilon).
                \end{equation}
                Observe that
                \[
                    M_{2\varepsilon}\ge \Psi_{2\varepsilon}(x_\varepsilon,y_\varepsilon)=
                    \Psi_{\varepsilon}(x_\varepsilon,y_\varepsilon)+
                    \dfrac{|x_\varepsilon-y_{\varepsilon}|^2}{2\varepsilon}=M_\varepsilon+
                    \dfrac{|x_\varepsilon-y_{\varepsilon} |^2}{2\varepsilon}
                \]
                implies
                \[
                    \dfrac{|x_\varepsilon-y_{\varepsilon}|^2}{\varepsilon}\le 
                    2(M_{2\varepsilon}-M_{\varepsilon})\to 0
                \]
                as $\varepsilon\to 0^+.$ Therefore, we have
                \begin{equation}\label{eq:cp3}
                    \lim_{\varepsilon \to 0^+} \dfrac{|x_\varepsilon-y_{\varepsilon}|^2}{\varepsilon}=0.
                \end{equation}
                
                Since $\overline{B_R}$ is compact, extracting a subsequence if necessary, we can assume that
                \begin{equation}\label{eq:cp4}
                    \lim_{\varepsilon \to 0^+} (x_\varepsilon,y_\varepsilon)\to ({\bar x},{\bar y})\in
                    \overline{B_R}\times\overline{B_R}.
                \end{equation}
                Moreover, by \eqref{eq:cp3}, ${\bar x}={\bar y}$ and
                \[
                \begin{array}{l}
                \displaystyle
                    M\le \overline{M}
                     =\lim_{\varepsilon\to0^+}\Psi_\varepsilon(x_\varepsilon,y_\varepsilon)
                    \\[10pt]
                     \displaystyle \qquad \le\limsup_{\varepsilon\to0^+} \big( u^g(x_\varepsilon)-v_g(y_\varepsilon) \big) \\[10pt]
                     \displaystyle \qquad
                     \le
                    u^g({\bar x})-v_g({\bar x})\le M.
                    \end{array}
                \]

                Thus $M=u^g({\bar x})-v_g({\bar x})$. This limit point $\bar x$ cannot be outside  $\overline{\Omega}$, since $M>0$,  and by Theorem \ref{condicion-borde.77}, it cannot be on $\partial \Omega$. Consequently, $\bar x \in \Omega$ and we may assume (without loss of generality) that
                \[
                    d_\varepsilon=\min \Big\{\dist(x_\varepsilon), 
                    \dist(y_\varepsilon )\Big\}>\frac{\dist({\bar x})}2>0
                \]
                provided $\varepsilon$ is small enough.
                
                On the other hand, by \eqref{eq:xeye}, for any $w\in \RR^N$ such that 
                 we have $$(x_\varepsilon+w,y_\varepsilon+w)\in\overline{B_R}\times\overline{B_R}$$ we have that
                \begin{equation}\label{eq:uvpositive}
                    0\le u(x_\varepsilon)-u^g(x_\varepsilon+w)
                    -v(y_\varepsilon)+v_b(y_\varepsilon+w).
                \end{equation}
                So
                $$
                    \phi_\varepsilon(x)\coloneqq 
                    v(y_\varepsilon)
                    +\dfrac{|x-y_\varepsilon |^2}{\varepsilon},
                    $$
                    and $$
                    \varphi_\varepsilon(y)\coloneqq
                    u(x_\varepsilon)
                    -\dfrac{|x_\varepsilon-y |^2}{\varepsilon} 
                $$
                are test functions for $u$ and $v$ at $x_\varepsilon$ and $y_\varepsilon,$ respectively. 
                Then, we have that
                 \begin{align*}
                    & E_\delta (u^g, \phi_\varepsilon, x_\varepsilon)\leq 0
                    \\
                    &\qquad  \mbox{ and } \\
& E_\delta (v_g, \varphi_{\varepsilon}, y_\varepsilon)\geq 0,
                \end{align*}
                for all $\delta\in(0,d_\varepsilon).$
                
                At this point we have to choose two sequences of directions, one to
                approximate the supremum and another one for the infimum.
                As before, for the subsolution we are free to choose directions in the infimum;
                while for the supersolution we can choose a direction close to the supremum. 

By the definition of $E_\delta$, 
for each $h>0$ there exists $z_{\varepsilon,h}, \widetilde{z}_{\varepsilon,h} \in \mathbb{S}^{N-1} $ such that
\begin{align}
    \inf_{|z|=1}E_{z,\delta} (v_g, \varphi_\varepsilon, y_\varepsilon) + h &\geq E_{z_{\varepsilon,h},\delta} (v_g, \varphi_\varepsilon, y_\varepsilon), \\
    \sup_{|z|=1}E_{z,\delta} (u^g, \phi_\varepsilon, x_\varepsilon) - h &\leq E_{\widetilde{z}_{\varepsilon,h},\delta} (u^g, \phi_\varepsilon, x_\varepsilon),
\end{align}
and thus, we get
\begin{equation} \label{EE}
\begin{array}{l}
\displaystyle    E_{z_{\varepsilon,h},\delta} (u^g, \phi_\varepsilon, x_\varepsilon) + E_{\widetilde{z}_{\varepsilon,h},\delta} (u^g, \phi_\varepsilon, x_\varepsilon) \geq -h
\\[10pt]
\displaystyle  
\qquad \mbox{ and } 
 \\[10pt]
\displaystyle    E_{z_{\varepsilon,h},\delta} (v_g, \varphi_\varepsilon, y_\varepsilon) + E_{\widetilde{z}_{\varepsilon,h},\delta} (v_g, \varphi_\varepsilon, y_\varepsilon) \leq h,
\end{array}
\end{equation}
for any $\delta\in(0,d_\varepsilon).$

Now, our goal is to obtain upper estimates for the differences
\[
\begin{array}{l}
\displaystyle
    E_{z_{\varepsilon,h},\delta}(u^g, \phi_\varepsilon, x_\varepsilon)- 
    E_{z_{\varepsilon,h},\delta} (v_g, \varphi_{\varepsilon}, y_\varepsilon) 
    \\[10pt] \displaystyle \qquad \text{and} \quad 
    \\[10pt] \displaystyle  E_{\widetilde{z}_{\varepsilon,h},\delta}(u^g, \phi_\varepsilon, x_\varepsilon)- 
    E_{\widetilde{z}_{\varepsilon,h},\delta} (v_g, \varphi_{\varepsilon}, y_\varepsilon)
    \end{array}
\]
for $0<\delta<\tfrac{\dist(\bar x )}2$. In this way, when we substract the second expression to the first one in \eqref{EE}, we get a contradiction, provided we controlled the differences properly. 

We can 
assume that $$z_{\varepsilon,h} \to z_0,\quad \widetilde{z}_{\varepsilon,h} \to \widetilde{z}_0 $$
taking a subsequence 
if necessary. 		

To estimate the first difference, let us write
\begin{equation}\label{EE.89}
    \begin{array}{l}
        \displaystyle 
            E_{z_{\varepsilon,h},\delta} (u^g, \phi_\varepsilon, x_\varepsilon)
                =I_1(\delta,\varepsilon,h)+I_2(\delta,\varepsilon,h)+
                I_3(\varepsilon,h), \\[15pt]
        \displaystyle 
                E_{z_{\varepsilon,h},\delta} (v_g, \varphi_{\varepsilon}, y_\varepsilon)
            =J_1(\delta,\varepsilon,h)+J_2(\delta,\varepsilon,h)+
            J_3(\varepsilon,h),
    \end{array}
\end{equation}
where
\[
    \begin{array}{l}
        \displaystyle 
        I_1(\delta,\varepsilon,h)\coloneqq \int_{-\delta}^\delta 
            \frac{
                \phi_{\varepsilon}(x_{\varepsilon}+t z_{\varepsilon,h})-
                    \phi_{\varepsilon}(x_{\varepsilon})}{|t|^{1+2s}}dt,
            \\[15pt]
        \displaystyle 
            I_2(\delta,\varepsilon,h)
            \coloneqq \int_{A_\delta^{z_{\varepsilon,h}}(x_\varepsilon)} 
            \frac{u^g(x_{\varepsilon}+t z_{\varepsilon,h})-u(x_\varepsilon)}{|t|^{1+2s}}dt,
        \\[15pt]
        \displaystyle
        I_3(\varepsilon,h)
        \coloneqq \int_{\R\setminus L_{z_{\varepsilon,h}}(x_\varepsilon)} 
        \frac{g(x_{\varepsilon}+t z_{\varepsilon,h})-u(x_\varepsilon)}{|t|^{1+2s}}dt,
    \end{array}
\]
and
\[
    \begin{array}{l}
            \displaystyle
            J_1(\delta,\varepsilon,h)\coloneqq \int_{-\delta}^\delta \frac{\varphi_{\varepsilon}
            (y_{\varepsilon}+t z_{\varepsilon,h})-
            \varphi_{\varepsilon, \mu}(y_\varepsilon)}{|t|^{1+2s}}dt,
            \\[15pt]
         \displaystyle
            J_2(\delta,\varepsilon,h)
            \coloneqq\int_{A_\delta^{z_{\varepsilon,h}}(y_\varepsilon)} 
        \frac{v_g(y_\varepsilon+t z_{\varepsilon,h})-v(y_\varepsilon)}{|t|^{1+2s}}dt,
        \\[15pt]
        \displaystyle
        J_3(\varepsilon,h)
        \coloneqq
        \int_{\R\setminus
        L_{z_{\varepsilon,h}}(y_\varepsilon)} \frac{g(y_\varepsilon 
        +t z_{\varepsilon,h})-v(y_\varepsilon)}{|t|^{1+2s}}dt,
    \end{array}
\]
with $A_\delta^{z}(x)\coloneqq L^R_{z}(x_\varepsilon) \setminus (-\delta,\delta)$, $L^R_z(x) =  \{t \in \mathbb{R}: x + tz \in B_R\}$.

We first observe that there is a positive constant $C$ independent of $\delta,$ 
$\varepsilon$ and $h$ such that
\[
    \max \Big\{|I_1(\delta,\varepsilon,h)|, |J_1(\delta,\varepsilon,h)|\Big\}
    \le C\dfrac{\delta^{2-2s}}{\varepsilon}.
\]

For the estimate of $I_2(\delta,\varepsilon,h)-J_2(\delta,\varepsilon,h),$ 
we use that $x_\varepsilon,y_\varepsilon \to \bar x$ and $z_{\varepsilon,h} \to z_0$ to get,
\[
    \mathbf{1}_{L^R_{z_{\varepsilon,h}}(x_\varepsilon)}, \mathbf{1}_{L^R_{z_{\varepsilon,h}}
    (y_\varepsilon)} \to 
    \mathbf{1}_{L^R_{z_0}(\bar x)}\quad \mbox{ a.e. as }\quad\varepsilon,h \to 0.
\]
Hence,
\[\begin{aligned}
|\mathbf{1}_{L^R_{z_{\varepsilon,h}}(x_\varepsilon)}- \mathbf{1}_{L^R_{z_{\varepsilon,h}}(x_\varepsilon) \cap L^R_{z_{\varepsilon,h}}(y_\varepsilon) }| &\to 
    0\quad \mbox{ a.e. as }\quad\varepsilon,h \to 0, \\
|\mathbf{1}_{L^R_{z_{\varepsilon,h}}(y_\varepsilon)}- \mathbf{1}_{L^R_{z_{\varepsilon,h}}(x_\varepsilon) \cap L^R_{z_{\varepsilon,h}}(y_\varepsilon) }| &\to 
0\quad \mbox{ a.e. as }\quad\varepsilon,h \to 0.
\end{aligned}\]
Then, we divide the integrals $I_2$ and $J_2$ further
\[\begin{aligned}
I_2(\delta,\varepsilon,h)
&= \int_{\mathbb{R}\backslash (-\delta, \delta)} \mathbf{1}_{L^R_{z_{\varepsilon,h}}(x_\varepsilon)} \frac{u^g(x_{\varepsilon}+t z_{\varepsilon,h})-u(x_\varepsilon)}{|t|^{1+2s}}dt \\
&= \int_{\mathbb{R}\backslash (-\delta, \delta)} \Big(\mathbf{1}_{L^R_{z_{\varepsilon,h}}(x_\varepsilon)} - \mathbf{1}_{L^R_{z_{\varepsilon,h}}(x_\varepsilon) \cap L^R_{z_{\varepsilon,h}}(y_\varepsilon) } \Big) \frac{u^g(x_{\varepsilon}+t z_{\varepsilon,h})-u(x_\varepsilon)}{|t|^{1+2s}}dt \\
&\qquad + \int_{\mathbb{R}\backslash (-\delta, \delta)} \mathbf{1}_{L^R_{z_{\varepsilon,h}}(x_\varepsilon) \cap L^R_{z_{\varepsilon,h}}(y_\varepsilon)} \frac{u^g(x_{\varepsilon}+t z_{\varepsilon,h})-u(x_\varepsilon)}{|t|^{1+2s}}dt,
\end{aligned}\]
\[\begin{aligned}
J_2(\delta,\varepsilon,h)
&= \int_{\mathbb{R}\backslash (-\delta, \delta)} \mathbf{1}_{L^R_{z_{\varepsilon,h}}(y_\varepsilon)} \frac{v_g(y_\varepsilon+t z_{\varepsilon,h})-v(y_\varepsilon)}{|t|^{1+2s}}dt \\
&= \int_{\mathbb{R}\backslash (-\delta, \delta)} \Big(\mathbf{1}_{L^R_{z_{\varepsilon,h}}(y_\varepsilon)} - \mathbf{1}_{L^R_{z_{\varepsilon,h}}(x_\varepsilon) \cap L^R_{z_{\varepsilon,h}}(y_\varepsilon) } \Big) \frac{v_g(y_\varepsilon+t z_{\varepsilon,h})-v(y_\varepsilon)}{|t|^{1+2s}}dt \\
&\qquad + \int_{\mathbb{R}\backslash (-\delta, \delta)} \mathbf{1}_{L^R_{z_{\varepsilon,h}}(x_\varepsilon) \cap L^R_{z_{\varepsilon,h}}(y_\varepsilon)} \frac{v_g(y_\varepsilon+t z_{\varepsilon,h})-v(y_\varepsilon)}{|t|^{1+2s}}dt.  
\end{aligned}\]

Now, we observe that
the terms in which we have the difference of characteristic functions go to zero, using the boundedness of $u^g,v_g$ and dominated convergence theorem. The difference of the terms that have a segment in common is negative thanks to \eqref{eq:uvpositive}. 

Collecting all these bounds we get 
\[\limsup_{\delta, \varepsilon, h \to 0^+}I_2(\delta,\varepsilon,h) - J_2(\delta,\varepsilon,h) \leq 0.\]

Finally, we observe that we have
\[
    \mathbf{1}_{\mathbb{R} \setminus L^R_{z_{\varepsilon,h}}(x_\varepsilon)}, 
    \mathbf{1}_{\mathbb{R} \setminus L^R_{z_{\varepsilon,h}}(y_\varepsilon)} 
    \to \mathbf{1}_{\mathbb{R} \setminus L^R_{z_0}(\bar x)} \quad\mbox{ a.e. as }\quad\varepsilon,h \to 0,
\]
and hence, arguing similarly to the previous case and using that $g$ is a bounded continuous function, we obtain
\[
    \lim_{\varepsilon,h\to0}I_3(\varepsilon,h)-J_3(\varepsilon,h)
    = -M \int_{L_{z_0}(\bar x)\cap (\mathbb{R}^N \setminus \Omega)} \frac{dt}{|t|^{1+2s}}. \quad 
\]
The negative term, $-M$, comes from the fact that $$\lim_{\varepsilon \to 0^+} (u(x_\varepsilon)-v(x_\varepsilon))= M.$$
Therefore, letting first $\delta \to 0$, then $\varepsilon \to 0 $, 
and $h \to 0,$ we get
\begin{equation}\label{eq:final}
\begin{array}{l}
\displaystyle
    I_1(\delta,\varepsilon,h)-J_1(\delta,\varepsilon,h)+I_2(\delta,\varepsilon,h)-
    J_2(\delta,\varepsilon,h)+
                I_3(\varepsilon,h)-J_3(\varepsilon,h) \\[6pt]
                \displaystyle \qquad 
                \to -M \int_{L_{z_0}(\bar x)\cap (\mathbb{R}^N \setminus \Omega)}
                \frac{dt}{|t|^{1+2s}}.
                \end{array}
\end{equation}
We conclude that
\[
\begin{array}{l}
\displaystyle 
 \lim_{h \to 0} \lim_{\varepsilon \to 0}  \lim_{\delta\to 0} 
 \Big\{ E_{z_h,\delta} (u^g, \Psi_{\varepsilon, \mu}(\cdot,\bar y), \bar x)- 
    E_{z_h,\delta} (v_g, \Psi_{\varepsilon, \mu}( \bar x, \cdot)  \Big\}
   \\[10pt]
   \qquad \displaystyle  \leq   -M \int_{L_{z_0}(\bar x)\cap (\mathbb{R}^N \setminus \Omega)} \frac{dt}{|t|^{1+2s}}<0.
   \end{array}
\]

For the supremum part the limit can be bounded by exactly the same
quantity (with a possible different limit direction $\widetilde{z}_0$),
\[
\begin{array}{l}
\displaystyle 
 \lim_{h \to 0} \lim_{\varepsilon \to 0}  \lim_{\delta\to 0} 
 \Big\{ E_{\widetilde{z}_h,\delta} (u^g, \Psi_{\varepsilon, \mu}(\cdot,\bar y), \bar x)- 
    E_{\widetilde{z}_h,\delta} (v_g, \Psi_{\varepsilon, \mu}( \bar x, \cdot)  \Big\}
   \\[10pt]
   \qquad \displaystyle  \leq   -M \int_{L_{\widetilde{z}_0}(\bar x)\cap (\mathbb{R}^N \setminus \Omega)} \frac{dt}{|t|^{1+2s}}<0.
   \end{array}
\]
 The fact that we get a negative bound came from having 
  $$\lim_{\varepsilon \to 0^+} (u(x_\varepsilon)-v(x_\varepsilon))= M,$$
 and not from the fact that we analyze the infimum. 

Thus, substracting both expressions of \eqref{EE} and taking limits we obtain
we get
\begin{align*}
0 = \lim_{h\to 0} - 2h &\leq  \lim_{h \to 0} \lim_{\varepsilon \to 0}  \lim_{\delta\to 0} \Big\{ E_{z_h,\delta} (u^g, \Psi_{\varepsilon, \mu}(\cdot,\bar y), \bar x)- 
    E_{z_h,\delta} (v_g, \Psi_{\varepsilon, \mu}( \bar x, \cdot) \Big\} \\[6pt]
    & \qquad +  \lim_{h \to 0} \lim_{\varepsilon \to 0}  \lim_{\delta\to 0} 
    \Big\{ E_{\widetilde{z}_h,\delta} (u^g, \Psi_{\varepsilon, \mu}(\cdot,\bar y), \bar x)- 
    E_{\widetilde{z}_h,\delta} (v_g, \Psi_{\varepsilon, \mu}( \bar x, \cdot) \Big\} \\[6pt]
    & \leq  -M \int_{L_{z_0}(\bar x)\cap (\mathbb{R}^N \setminus \Omega)} \frac{dt}{|t|^{1+2s}}
    -M \int_{L_{\widetilde{z}_0}(\bar x)\cap (\mathbb{R}^N \setminus \Omega)} \frac{dt}{|t|^{1+2s}} <0.
\end{align*}
and we end up with the desired contradiction.
\end{proof}

 \begin{remark} \label{rem-com-general} This proof can be extended to deal with the problem involving 
     the trace fractional laplacian $(-\Delta)_{tr}^s$. As before we just observe that, for the
     terms that involve the intermediate fractional eigenvalues,
      \begin{equation} \label{fract-eigen.99678}
       \Lambda_i^s u (x) = \max_{{dim} (S) = N-i+1}
       \min_{\substack{z \in S, \\ |z| =1}}  \left\{  \int_{\mathbb{R}} 
                             \frac{u(x+tz)-u(x)}{|t|^{1+2s}} 
                             \, dt
                         \right\},
       \end{equation}
       with $i=2,...,N-1$, we can choose a subspace and then a direction that almost reach the associated quantity
       $$
       \sup_{{dim} (S) = N-i+1}
       \inf_{ \substack{z \in S, \\ |z| =1}}  
       E_{z, \delta_\varepsilon}(u, \omega_\varepsilon, x_\varepsilon).
       $$

With the same computations
one can obtain a comparison principle (and then existence and 
uniqueness of solutions, see below)  for $C^2$ domains and continuous and bounded exterior
 data, $g$,
 for problems of the form
 \begin{equation} \label{Main-eq-section.4466} 
        \left\{
             \begin{array}{ll}
                \displaystyle \sum_{i=1}^N a_i \Lambda_i^s u (x) =0, \qquad & x \in \Omega, \\[6pt]
                         u (x) = g (x), \qquad & x \in \mathbb{R}^N\setminus \Omega,
             \end{array}
             \right.
         \end{equation}    
as long as $a_1>0$, $a_N>0$ and $a_i \geq 0$, $i=2,...,N-1$
with $g$ continuous and bounded and $\Omega$ a $C^2$ bounded domain.

 If only one of the maximum or minimum fractional eigenvalues
 ($\Lambda_1^s u$ or $\Lambda_N^s u$)
 is involved in the operator, then the proof still works (as in \cite{BdpQR}) 
 with the extra assumption that the domain is strictly convex. This assumption
 (strict convexity) ensures that, close to the boundary, the line in any direction
 reaches the boundary close to the nearest point on the boundary. 
 
  We refer to Section \ref{sect-extension}
 for
extra comments on extensions of our results. 
           \end{remark}

\subsection{Existence and uniqueness of a solution} 
This part is standard in the viscosity theory once one has at hand
a comparison principle, but we include the details here for
completeness.

	Now our goal is to show existence and uniqueness of a solution to
	\eqref{Main-eq-section} (the equation involves 
	only $\Lambda_1^s u$ and $\Lambda_N^s u$) and \eqref{Main-eq-section.44}
	(the equation is given by the trace fractional Laplacian, that is, the sum of the 
	fractional eigenvalues). The same proof works for any of the operators
	that involve fractional eigenvalues as long as we have a comparsion principle,
	see Remark \ref{rem-com-general}.
	
	The proof of existence is by
	using Perron's method and uniqueness is immediate
	from the comparison principle.

	\begin{theorem}\label{EyU.77.99} 
		Assume that $g\in C(\R^N\setminus {\Omega})$ is bounded 
		and $\Omega$ is a bounded $C^2-$domain.  
		Then, there is a unique viscosity solution $u$ to 
		\eqref{Main-eq-section} or to \eqref{Main-eq-section.44} 
		in $\Omega$, in the sense of Definition \ref{defsol.44}.
		This unique solution is continuous in $\overline{\Omega}$ and the datum $g$ is taken with continuity, 
		that is, $u|_{\partial \Omega} = g|_{\partial \Omega}$.
	\end{theorem}

	\begin{proof} Again we use ideas from \cite{BarChasImb} to obtain the existence of 
		a viscosity solution to our Dirichlet problem.

		Existence of $u,v\colon \mathbb{R}^N\to\mathbb{R}$ that are a viscosity 
		subsolution and a viscosity supersolution 
		 in $\Omega$, in the sense of Definition \ref{defsol.44}, 
		follows easily taking large constants (here we are using that
		$g$ is bounded).

		We take a one-parameter family of continuous functions 
		$\psi_{\pm}^k\colon \R^N \to \R$ such that $\psi_+^k \geq \psi_-^k$ 
		in $\R^N$ and $\psi^k_{\pm} = g$ in $\R^N\setminus\Omega.$
		Then for all $k\in\N$ we consider the obstacle problem 
		\begin{equation}\label{obstacle}
			H_k(x, u)\coloneqq \min \Big\{ u(x) - \psi_-^k(x), 
			\max \big\{ u(x) - \psi_+^k(x), -\Lambda_1^s u (x) u(x) \big\} \Big\},
		\end{equation}
		for $x \in \R^N$,
		which is degenerate elliptic, that is, it satisfies the general assumption $(E)$ of 
		\cite{BarImb}. It has $\pm\|g\|_\infty$ as viscosity super and subsolution. 
		Recall that these viscosity supersolution and subsolution do not depend on the $L^\infty$ bounds 
		of $\psi^k_{\pm}$. Then, in view of the general Perron's method given in \cite{BarImb} 
		for problems in $\R^N$, since condition $(E)$ holds, we conclude the existence of a continuous bounded
		viscosity solution $u_k$ to~\eqref{obstacle} for each $k$ and, in addition, 
		this family of solutions is equal to $g$ in $\R^N\setminus\Omega$ for all $k$. Moreover,
		we have that
        \[\|u_k\|_{L^\infty (\Omega)} \leq \|g\|_{L^\infty (\mathbb{R}^N \setminus \Omega)} \quad \text{and} \quad u_k = g \in \RR^N\setminus \Omega\]
        for all $k \in \mathbb{N}$.

		Then, we consider $\psi_{\pm}^k$ in such a way $\psi_{\pm}^k(x) \to \pm \infty$ as $k \to \infty$ 
		for all $x \in \Omega$ and denoting
		\begin{equation*}
			\bar u(x) = \limsup \limits_{k \to \infty, y \to x} u_k(y); \quad 
			\underline u(x) = \liminf \limits_{k \to \infty, y \to x} u_k(y),
		\end{equation*}
		which are well defined for all $x \in \R^N$, 
		we clearly have that $$\bar u \geq \underline u$$ in $\R^N$.
		Thus, we have $$\underline u = \bar u = g$$ in $ \R^N\setminus\Omega$ and 
		$\bar u$ and $\underline u$ respectively are a viscosity subsolution and a 
		viscosity supersolution 
		to our problem. Thus, by comparison we get
		$$\bar u \leq \underline u$$ in $\R^N$, and therefore we conclude that $\bar u$ and $\underline u$
		coincide and that
		$$u:=\bar u = \underline u$$
		is a continuous viscosity solution that satisfies the boundary condition 
		in the classical sense.

		Uniqueness of solutions follows from the comparison principle.
\end{proof}

The following corollary is immediate and we just estate it here
since it is a part of Theorem \ref{teo.compar.intro}.

\begin{corollary}
The comparison principle implies that if $g_1\geq g_2$ are boundary conditions of solutions $u_1$ and $u_2$ respectively, we have that $u_1\geq u_2$. 
\end{corollary}

\begin{proof} This follows from considering $u_2$ a supersolution for the boundary condition $g_1$ or considering $u_2$ a subsolution for the boundary condition $g_2$.
\end{proof}

Also as an immediate corollary of the comparison principle
we obtain continuous dependence of the solution
with respect to the exterior data, this gives another part of Theorem \ref{teo.compar.intro}.

\begin{corollary}
The comparison principle implies that 
the solution depends continuously
        on the exterior data; it holds that
        $$
        \| u_1 - u_2 \|_{L^\infty (\Omega)} \leq \| g_1 - g_2 \|_{L^\infty (\mathbb{R}^N \setminus \Omega)}.
        $$
\end{corollary}

\begin{proof} Just observe that 
$$u_2 (x) + \| g_1 - g_2 \|_{L^\infty (\mathbb{R}^N \setminus \Omega)} $$
is a supersolution to the problem with exterior datum $g_1$, hence, by comparison we get
$$u_1(x) \leq u_2 (x) + \| g_1 - g_2 \|_{L^\infty (\mathbb{R}^N \setminus \Omega)}.$$

In a similar way, we get that
$$u_1(x) \geq u_2 (x) - \| g_1 - g_2 \|_{L^\infty (\mathbb{R}^N \setminus \Omega)},$$
since the left hand side is 
a subsolution to the problem with exterior datum $g_1$.
\end{proof}

Now, our goal is to show that \eqref{Main-eq-section} has a strong maximum principle,
a viscosity subsolution to the problem can not attain 
the maximum inside $\Omega$ unless it is constant.

\begin{theorem}[Strong maximum principle]
Let $u$ be a subsolution to problem \eqref{Main-eq-section}
such that $u^g$ attains a maximum at some point $\bar x \in \Omega$. Then $u$ is constant. 
\end{theorem}

\begin{proof} We argue by contradiction.
Let us assume $u$ is not constant. That means there is some $x_0$ in $\Omega$ such that $u(\bar x)> u(x_0)$. Since $\bar x \in \Omega$, we can take $0<\delta<|\bar x - x_0|/2$ and $\phi(x) = u(\bar x)$ for $x \in \overline{B_\delta}(\bar x)$. We extend this $\phi$ to all $\RR^N$ in a smooth way. Since u is a subsolution, 
\[\inf_{z \in \mathbb{S}^{N-1}} E_{z,\delta} (u^g, \phi, \bar{x}) + \sup_{z \in \mathbb{S}^{N-1}} E_{z,\delta} (u^g, \phi, \bar{x}) \geq 0.\]
For the infimum, we can choose a direction freely, so let us choose $z_0=x_0-\bar x$. For the supremum, fix $z_k$ such that 
\[\sup_{z \in \mathbb{S}^{N-1}} E_{z,\delta} (u^g, \phi, \bar{x}) \leq E_{z_k,\delta} (u^g, \phi, \bar{x}) + \frac{1}{k}. \]
This way we get 
\[-\frac{1}{k} \leq E_{z_0,\delta} (u^g, \phi, \bar{x}) +  E_{z_k,\delta} (u^g, \phi, \bar{x}). \]
Both expressions lack the integral where the test function appears due to the fact that 
the test function is constant. What we have to estimate is 
\[I(z) = \int_{\RR \setminus(-\delta, \delta)} \frac{u^g(\bar x + t z) - u(\bar x)}{|t|^{1+2s}}dt\]
For every $z$ we have that $I(z) \leq 0$ thanks to the fact that the difference inside the integral
is nonpositive. But for $z_0$, we get an estrict inequality. The reason behind this is that upper continuity grants the existence of a ball centered at $x_0$ of radius sufficiently small so that for every $y$ in that ball, $u(y)<u(\bar x)$. This yields that for a certain interval centered around $t=1$, the difference $u(\bar x + tz_0)-u(\bar x)<0$, so we can conclude that $I(z_0)<0$. Hence, 
\[-\frac{1}{k} \leq E_{z_0,\delta} (u^g, \phi, \bar{x}) +  E_{z_k,\delta} (u^g, \phi, \bar{x}) < 0\]
Taking the limit $k \to +\infty$ we get $0<0$, arriving at the desired contradiction. 
\end{proof}

With the same idea one can show that a supersolution to the problem \eqref{Main-eq-section}
that attains a minimum at some point $\bar x \in \Omega$ must be constant.

\begin{remark} In our proof of the strong maximum principle
we need to choose directions. When we deal with subsolutions  
we can choose the direction that is involved in the infimum and
hence any non-constant solution to an equation given in terms of a sum of fractional  
eigenvalues can not attain an interior maximum provided $\Lambda_1^s u$
appears in the operator. Analogously, when $\Lambda_N^s u$ appears in the operator
non-constant solutions can not have interior minima.
\end{remark}

\begin{corollary} \label{corol55}
For an exterior datum $g\gneqq 0$ we have that
the corresponding solution to \eqref{Main-eq-section} is strictly positive in $\Omega$, $u(x)>0$,
$x \in \Omega$.
\end{corollary}

Finally, we show a strong comparison principle, provided that the exterior
data verify that $g_1\geq g_2$ and there exists a point $\hat{x}$ in every line
that passes trough $\Omega$ such that $g_1 (\hat{x}) > g_2 (\hat{x})$.

\begin{theorem}[Strong comparison principle]
\label{teo.string}
Assume $g_1,g_2 \in C(\RR^N \setminus \Omega)$ with 
$g_1\geq g_2$ and
in every line
that passes trough $\Omega$ there exists a point $\hat{x} \in \mathbb{R}^N \setminus \Omega$  such that $g_1 (\hat{x}) > g_2 (\hat{x})$. Let $u_1$,
$u_2$ be solutions of our problem with boundary conditions $g_1$ and $g_2$ respectively. 
Then, $u_1 > u_2 $ in $\Omega$.
\end{theorem}

\begin{proof} 
We already know, thanks to the comparison principle, that $u_1 \geq u_2$. 
Arguing by contradiction, assume that there exists $x_0 \in \Omega$ such that $u_1(x_0) = u_2(x_0)$.

Set $$\eta = \frac{d(x_0, \partial \Omega)}{2},$$ and define 
\[\Omega_{-\eta} := \{x \in \Omega: d(x, \partial \Omega)>\eta\}.\]
Since $x_0 \in \Omega_{-\eta}$, we know that 
\[\inf \{(u_1-u_2)(x) : x \in \overline{\Omega}_{-\eta}\} = 0.\]

Next, we proceed to construct an auxiliary function that will yield a test function for $u_1$ from below and a test function for $u_2$ from above like the one used in the proof of Theorem \ref{teo:cp}. We aim to use the fact that $u_1$ is a viscosity supersolution and $u_2$ a viscosity subsolution. 

Let 
\[\psi_\varepsilon(x,y) = u_1(x) - u_2(y) + \frac{|x-y|^2}{\varepsilon}\]
and 
\[S_\varepsilon := \inf\{\psi_\varepsilon(x,y): (x,y) \in \overline{\Omega}_{-\eta} \times \overline{\Omega}_{-\eta}\}. \]
It is simple to observe, following similar steps to the ones in the proof of Theorem \ref{teo:cp}, that $S_{\varepsilon_2} \leq S_{\varepsilon_1}$ for $\varepsilon_1 \leq \varepsilon_2$ and $S_\varepsilon \leq 0$. Using the continuity of $\psi_\varepsilon$ we know that there exists $(x_\varepsilon, y_\varepsilon) \in \overline{\Omega}_{-\eta} \times \overline{\Omega}_{-\eta}$ such that $\psi_\varepsilon (x_\varepsilon, y_\varepsilon) = S_\varepsilon$. This implies
\[\lim_{\varepsilon \to 0^+} \psi_\varepsilon(x_\varepsilon, y_\varepsilon) = \lim_{\varepsilon \to 0^+} S_\varepsilon = \overline{S} \leq 0 . \]
Taking a subsequence, we may assume $x_\varepsilon$ and $y_\varepsilon$ converge to $\widetilde{x}$ and $\widetilde{y}$, and we know that $\widetilde{x} = \widetilde{y}$ thanks to 
\[\lim_{\varepsilon \to 0^+}\frac{|x_\varepsilon-y_\varepsilon|^2}{\varepsilon} = 0.\]
Hence, we get
\[0 \geq \overline{S} = \lim_{\varepsilon \to 0^+} \psi_\varepsilon(x_\varepsilon, y_\varepsilon) = \lim_{\varepsilon \to 0^+} \big( u(x_\varepsilon) - v(y_\varepsilon) \big) \geq 0.\]
This implies $x_\varepsilon$ and $y_\varepsilon$ converge to a minimum point of $u_1-u_2$, and 
since this point belongs to $\overline{\Omega}_{-\eta}$, it follows that it is inside $\Omega$. 
With no loss of generality we will assume $\widetilde{x} = x_0$. 

Thanks to the previous arguments, we have found test functions 
\begin{align*}
    \phi_\varepsilon(x)&\coloneqq 
    u_2 (y_\varepsilon)
    -\dfrac{|x-y_\varepsilon |^2}{\varepsilon},\\[6pt]
    \varphi_\varepsilon(y)&\coloneqq
    u_1(x_\varepsilon)
    +\dfrac{|x_\varepsilon-y |^2}{\varepsilon} 
\end{align*}
for $u_1$ and $u_2$ at the points $x_\varepsilon$ and $y_\varepsilon$ respectively. We now use the fact that $u_1$ is a viscosity supersolution and $u_2$ a subsolution so that
\[E_\delta (u_1, \phi_\varepsilon, x_\varepsilon) \geq 0\quad \text{and}  \quad E_\delta (u_2, \varphi_\varepsilon, y_\varepsilon) \leq 0. \]
Now we use the same strategy as in Theorem \ref{teo:cp} to get rid of the infimum and the supremum of the expressions. For each $h>0$ there exists $z_{\varepsilon,h}, \widetilde{z}_{\varepsilon,h} \in \mathbb{S}^{N-1} $ such that
\begin{align}
    \inf_{|z|=1}E_{z,\delta} (u_1, \phi_\varepsilon, x_\varepsilon) + h &\geq E_{z_{\varepsilon,h},\delta} (u_1, \phi_\varepsilon, x_\varepsilon), \\
    \mbox{and} \\
    \sup_{|z|=1}E_{z,\delta} (u_2, \varphi_\varepsilon, y_\varepsilon) - h &\leq E_{\widetilde{z}_{\varepsilon,h},\delta} (u_2, \varphi_\varepsilon, y_\varepsilon).
\end{align}
Then,
\begin{equation} \label{EE.22}
    \begin{array}{l}
    \displaystyle    E_{z_{\varepsilon,h},\delta} (u_1, \phi_\varepsilon, x_\varepsilon) + E_{\widetilde{z}_{\varepsilon,h},\delta} (u_1, \phi_\varepsilon, x_\varepsilon) \leq h
    \\[10pt]
    \displaystyle  
    \qquad \mbox{ and } 
     \\[10pt]
    \displaystyle    E_{z_{\varepsilon,h},\delta} (u_2, \varphi_\varepsilon, y_\varepsilon) + E_{\widetilde{z}_{\varepsilon,h},\delta} (u_2, \varphi_\varepsilon, y_\varepsilon) \leq -h,
    \end{array}
    \end{equation}
Here the contradiction will follow from the fact that, 
\[\lim_{\varepsilon, h \rightarrow 0^+} \Big( g_1(x_\varepsilon+tz_{\varepsilon,h})-g_2(y_\varepsilon+tz_{\varepsilon,h}) \Big)= g_1(x_0 + tz_0)-g_2(x_0 + tz_0).\]
By hypothesis, for some $t_0$, $(x_0+t_0z_0) \not \in \Omega$ we get 
\[g_1(x_0 + tz_0)-g_2(x_0 + tz_0)>0.\]
Now, using the continuity of the exterior data $g_1$ and $g_2$, the strict inequality holds on some interval $I$. Doing exactly the same computations on the other set of directions, after taking limits, we find the desired contradiction
\[0<\int_I \frac{g_1(x_0 + tz_0)-g_2(x_0 + tz_0)}{|t|^{1+2s}}\, dt + \int_{\widetilde{I}} \frac{g_1(x_0 + t\widetilde{z}_0)-g_2(x_0 + t\widetilde{z}_0)}{|t|^{1+2s}}\, dt \leq 0.\]
Therefore, we conclude that $u_1 > u_2$ in $\Omega$. 
\end{proof}

\begin{remark}
As a simple example where Theorem \ref{teo.string} can be applied, one can take functions $g_1$, $g_2$ that verify 
$g_1\geq g_2$ with $g_1(x) > g_2(x)$ for $x$ inside an annulus $r < |x| <R$
and a domain $\Omega \subset B_r (0)$.
\end{remark}

\begin{remark}
To obtain that $u_1 > u_2$
we need to assume that in every line that passes trough $\Omega$ 
there exists a point $\hat{x} \in \mathbb{R}^N \setminus \Omega$ 
such that $g_1 (\hat{x}) > g_2 (\hat{x})$. 
This condition is necessary, in fact, in the next section
we will construct data with 
$g_1\gneqq  g_2$  in $ \mathbb{R}^N \setminus \Omega$
and such that the corresponding solutions coincide, 
$u_1 \equiv  u_2$ in $\Omega$.
\end{remark}

\section{The trace fractional Laplacian is nonlinear} \label{sect-nonlinear}

Our goal in this section is to show that the trace fractional Laplacian
and the mid-range fractional Laplacian
are nonlinear operators. Since in $\mathbb{R}^2$
both operators coincide up to a constant we perform
the argument for $N=2$.

\begin{theorem}
The problems \eqref{Main-eq-intro} and \eqref{Main-eq-section.iiii} are nonlinear problems. 
There exist
         data $g_1$, $g_2$ such that
         \begin{equation} \label{compar.7799}
        g_1\gneqq  g_2 \mbox{ in } \mathbb{R}^N \setminus \Omega,
       \mbox{ but } 
u_1 \equiv  u_2  \mbox{ in }\Omega. 
                \end{equation}
\end{theorem}

\begin{proof}
    We will argue in $\mathbb{R}^2$ and denote a point as $(x,y)$. To begin, we consider $a(x)$ a 
    viscosity solution in $\mathbb{R}$ to 
    \[(-\Delta)^{s}a(x) = - p.v. \int_{\mathbb{R}} \frac{a(x+t)-a(x)}{|t|^{1+2s}}\, dt = 1,\]
    for $x\in (-2,2)$.
    This viscosity solution is going to be a strong solution in $(-1,1)$, and in fact we can assume it is at least $C^2$, so we can drop the principal value.

    Define in $\mathbb{R}^2$ the function 
    \[u(x,y) = a(x) - a(y).\]
    Observe that, for a unitary vector $v=(v_x,v_y)$, we have
    \[\begin{aligned}
       & \int_{\RR} \frac{u(x+tv_x, y+tv_y)-u(x,y)}{|t|^{1+2s}}\, dt \\
       &= \int_{\RR} \frac{a(x+tv_x)-a(x)}{|t|^{1+2s}}\, dt - \int_{\RR} \frac{a(y+tv_y)-a(y)}{|t|^{1+2s}}\, dt \\
        & = |v_x|^{2s} \int_{\RR} \frac{a(x+\widetilde{t})-a(x)}{|\widetilde{t}|^{1+2s}}\, d\widetilde{t} - |v_y|^{2s} \int_{\RR} \frac{a(y+\widetilde{t})-a(y)}{|\widetilde{t}|^{1+2s}}\, d\widetilde{t} \\
        & = |v_x|^{2s} - |v_y|^{2s}.
    \end{aligned}\]
We can easily check that the supremum and the infimum of this expression are achieved at $v=(1,0)$ and $v=(0,1)$ respectively. Thus, this particular $u(x,y)$ satisfies our equation in $\Omega = (-1,1)\times (-1,1)$. We have, 
\[
\begin{array}{l}
\displaystyle 
(-\Delta)_{tr}^s u(x,y) = -\inf_{z\in \mathbb{S}^{1}} \int_{\mathbb{R}} 
\frac{u(x+tz)-u(x)}{|t|^{1+2s}}\, dt  \\[10pt]
\qquad \qquad  \qquad \qquad \displaystyle  -\sup_{z\in \mathbb{S}^{1}} \int_{\mathbb{R}}\frac{u(x+tz)-u(x)}{|t|^{1+2s}}\, dt 
\\[10pt] \qquad \qquad \displaystyle
= -1+1 = 0.
\end{array} 
\]
Consider the Dirichlet problem 
\[\left\{\begin{aligned}
    (-\Delta)_{tr}^s w(x,y) &= 0, \quad &&\text{ for } (x,y) \in B_1(0) \\
    w(x,y) &= u(x,y), \quad &&\text{ for } (x,y) \in \RR^2 \setminus B_1(0).
\end{aligned}\right.\]
This problem has the unique solution that we already constructed, $w(x,y) = u(x,y)$
(uniqueness comes from the comparison principle). 

Now, if we adequately add a little perturbation far from the origin 
supported close to a point $(\hat x, \hat x)$ on the diagonal to obtain a new exterior datum, we will 
get that for this different exterior datum we get the same solution. 

Let $f(x,y)$ be a radially non-increasing nonnegative and nontrivial cut-off function such that $f(x,y) = 1$ for $(x,y) \in B_r(\hat x, \hat y)$ and $f(x,y) = 0$ for $(x,y) \in \RR^2 \setminus B_r(\hat x, \hat x)$ for some far away point $(\hat x, \hat x)$ in a diagonal and some small $r$. 
Now, let us take $\widetilde{g} (x,y) = u(x,y) + \varepsilon f(x,y)$ as exterior datum.
Then, using calculations similar to those performed before we get that
\[
\begin{array}{l}
\displaystyle 
\int_{\RR} \frac{\widetilde{g} (x+tv_x, y+tv_y) - u(x,y)}{|t|^{1+2s}}\, dt  
\displaystyle = |v_x|^{2s} - |v_y|^{2s} + \varepsilon \chi_{\{|v_x| \simeq |v_y|\} }.
\end{array} \]
The $\varepsilon>0$, which we can assume to be small, is the effect of the perturbation $f$. Both the supremum and infimum are still going to be achieved for $v=(1,0)$ and $v=(0,1)$. This is quite immediate for the infimum, since $\varepsilon$ is positive. 
The supremum does not change because, if $|v_x| \sim |v_y|$, then in a diagonal direction the above expression will be something similar to $2\varepsilon$, which can be chosen smaller than $1$ and thus, $z=(1,0)$ achieves a greater value. 
Therefore, even after altering the exterior datum by adding a 
nonegative and nontrivial perturbation, $u(x,y)$ is still a solution
inside $B_1(0)$. 

Now, let us substract the two functions in order to obtain a nonnegative
and nontrivial exterior datum (the perturbation). If our operator was linear, from our previous results
we obtain that the solution with this exterior datum is strictly positive inside $B_1(0)$ 
(we use the strong maximum principle, Corollary \ref{corol55}).
This proves the
nonlinearity of the problem, since the difference of two solutions that coincide in 
$B_1(0)$ is not the solution for the difference of the exterior data.
\end{proof}

        \section{Limit as $s\nearrow 1$} \label{sect-lim-s}
        
        In this section, for a fixed $C^2$ domain $\Omega$ and a fixed 
        continuous and bounded exterior datum $g$ we study the limit as $s \nearrow 1$ of the solutions $u_s$
        (we make explicit the dependence of the solution in $s$ along this section). 
        
        Recall that the fractional eigenvalues are given by
        \begin{equation}\label{eq:eigenvalues.sect}
    \displaystyle \Lambda_k^s u (x) =
    \max_{{dim} (S) = N-k+1}
       \min_{ \substack{z \in S, \\ |z| =1}}
    \left\{
    c(s) \, p.v. \int_{\mathbb{R}} \frac{u(x+tz)-u(x)}{|t|^{1+2s}}\, dt
\right\}
\end{equation}
Here we will use the explicit constant given by
         \begin{equation} \label{cte.99}
         c(s) = 
         \frac{2^{2s}s\Gamma  (s+1/2)}{\pi^{1/2} \Gamma  (1-s)},
         \end{equation}
         however, as we have mentioned in the introduction, any $c(s)$
         with $c(s) \sim (1-s)$ will give the same limit. 
        
Next, our goal is to show that $u_s$, the unique solution to
  \begin{equation} \label{Main-eq-lim} 
        \left\{
             \begin{array}{ll}
                \displaystyle(- \Delta)^s_{tr} (x) = - \sum_{i=1}^N \Lambda_i^s u (x) =0, \qquad & x \in \Omega, \\[6pt]
                         u (x) = g (x), \qquad & x \in \mathbb{R}^N\setminus \Omega,
             \end{array}
             \right.
         \end{equation}
         converges uniformly as $s\nearrow 1$ to the unique solution to 
                   \begin{equation} \label{Dirich-local-Lapla.sect} 
        \left\{
             \begin{array}{ll}
                \displaystyle \Delta u(x)  = 0 \qquad & x \in \Omega, \\[6pt]
                         u (x) = g (x) \qquad & x \in \partial \Omega.
             \end{array}
             \right.
         \end{equation}
         
         In addition, with the same arguments, we obtain that 
         when $u_s$ is the unique solution to 
          \begin{equation} \label{Main-eq-lim.iiii} 
        \left\{
             \begin{array}{ll}
                \displaystyle (-\Delta)^s_{mid} (x) =-\frac12 \Lambda_1^s u (x) -\frac12 \Lambda_N^s u (x) =0, \qquad & x \in \Omega, \\[6pt]
                         u (x) = g (x), \qquad & x \in \mathbb{R}^N\setminus \Omega,
             \end{array}
             \right.
         \end{equation}
         it holds that
         $$
         \lim_{s \nearrow 1} u_s = u
         $$
         in $C (\overline{\Omega})$ with the limit $u$ is given by the unique solution to
         the local problem
          \begin{equation} \label{Dirich-local-Lapla.N=2.sect} 
        \left\{
             \begin{array}{ll}
                \displaystyle \lambda_1 (D^2 u)(x) + \lambda_N (D^2 u) (x) = 0 \qquad & x \in \Omega, \\[6pt]
                         u (x) = g (x) \qquad & x \in \partial \Omega.
             \end{array}
             \right.
         \end{equation}

 First, we show that the half-relaxed upper limit of the $u_s$ is a subsolution 
 to the limit problem. 

\begin{theorem}
Let $u_s$ be viscosity subsolutions to our Dirichlet problem 
for the trace fractional
Laplacian
\eqref{Main-eq-lim}, and define $\overline{u}$ as the half-relaxed upper limit,
\[\overline{u}(x):= \limsup_{\subalign{s&\to 1^-\\y&\to x}}u_s(y).\] 
Then $\overline{u}$ is a viscosity subsolution to the Dirichlet problem
for the classical local Laplacian \eqref{Dirich-local-Lapla.sect}.
\end{theorem}

\begin{proof}
First, we notice that $u_s$ are uniformly bounded in $s$.
This fact can be easily obtained using the comparison principle 
and that $\overline{w} = \|g\|_\infty$ 
is a supersolution to \eqref{Main-eq-lim}.
Therefore, the half-relaxed upper limit,
\[\overline{u}(x):= \limsup_{\subalign{s&\to 1^-\\y&\to x}}u_s(y)\]
is well-defined and bounded. 
By definition, $\overline{u}$ is an upper semicontinuous function. 

Again, to simplify the notation, we will prove the result
in $\mathbb{R}^2$, that is, we take $N=2$, since in this case we have only two eigenvalues
(one is given by an infimum among directions and the other by a supremum). 
At the end of the proof we will add a few lines on how to treat the general case.

Choose $\bar x \in \Omega$ and $\phi$ a test function such that $\overline{u} - \phi$ 
attains a maximum at $\bar x$. We can assume that $(\phi - \overline{u})(\bar x) = 0$ and that $(\bar u - \phi)(x) > 0$ for $x \neq \bar x$. Since $u_s$ is upper semicontinuous by definition, $\psi_s=u_s - \phi$ reaches a maximum point at $x_s\in \overline{\Omega}$. By the definition of 
half-relaxed limit, we have
\[\overline{\psi}(x) \leq \sup \left\{\psi_s(y): 0<1-s < \frac{1}{k}, y \in B_{\frac{1}{k}}(x)\setminus \{x\}  \right\}, \quad \forall k \in \mathbb{N}\]
Choose some $k$. For this $k$, we can find $s_k,y_k$ such that 
\[
\begin{array}{l}
\displaystyle
\overline{\psi}(x) \leq \sup \left\{\psi_s(y): 0<1-s < \frac{1}{k}, y \in B_{\frac{1}{k}}(x)\setminus \{x\}  \right\} 
\\[10pt]
\qquad \displaystyle \leq \psi_{s_k}(y_k) + \frac{1}{k} \leq \psi_{s_k}(x_{s_k}) + \frac{1}{k} 
\end{array}
\]
We observe $s_k \nearrow 1$ when $k\to \infty$. By extracting a subsequence, we can assume $x_{s_k}\to x_0$ for some $x_0\in \overline{\Omega}$. Then,
\[\overline{\psi} (x) \leq \limsup_k \left( \psi_{s_k}(x_{s_k}) + \frac{1}{k} \right) = \overline{\psi}(x_0).  \]
Since the only maximum point of $\overline{\psi}$ was $\bar x$, we deduce $x_0 = \bar x$, and thus 
\[\lim_k x_{s_k} = \bar x\]
that is, the sequence of maximum points converge to the maximum of the half-relaxed limit. Now, we can use $\phi$ as a test function at $x_{s_k}$ for the subsolutions $u_{s_k}$. Let us choose $z_{k}, \widetilde{z}_{k}$ unitary vectors such that 
\begin{align}\label{eq:infsuperror}
    \inf_{z\in \mathbb{S}^{N-1}}E^s_{z,\delta} (u^g_s, \phi, x_s) &\geq E^s_{z_{k},\delta} (u^g_{s_{k}}, \phi, x_{s_{k}}) - \frac{1}{k} \\
    & \mbox{and}
    \\
    \sup_{z\in \mathbb{S}^{N-1}}E^s_{z,\delta} (u^g_s, \phi, x_s) &\leq E^s_{\widetilde{z}_{k},\delta} (u^g_{s_{k}}, \phi, x_{s_{k}}) + \frac{1}{k}
\end{align}
for some $0<\delta< \frac{\dist(\bar x, \partial \Omega)}{2} < \dist(x_s, \partial \Omega)$, a condition that we can assume with no loss of generality. 
Here $E^s_{z,\delta} (\cdot, \phi, x)$
is given as in Section \ref{sect-compar} (without multiplying by $c(s)$ the corresponding integrals). 
Notice that here we 
write $E^s_{z,\delta} (\cdot, \phi, x)$ to make explicit that the operator depends on $s$. 

Now, as before, we write
\[ E^s_{z_{k},\delta} (u^g_s, \phi, x_{s_k}) = I^{1,s}_{z_{k}, \delta}(\phi, x_{s_k})+I^{2,s}_{z_{k}, \delta}(u_s, x_{s_k})\]
with
\[
\begin{array}{l}
    \displaystyle 
        I^{1,s}_{z_k, \delta}(\phi, x_{s_k})
        = \int_{-\delta}^\delta \frac{\phi (x_{s_k}+t z_{k})-\phi(x_{s_k})}{|t|^{1+2s}}dt,\\[15pt]
    \displaystyle  
    \qquad \qquad \mbox{ and}
    \\[15pt]
    \displaystyle  
        I^{2,s}_{z_k, \delta}(u, x_{s_k}) = \int_{\RR \setminus (-\delta,\delta)} 
            \frac{u^g_s( x_{s_k}+t z_{k})-u_s(x_{s_k})}{|t|^{1+2s}}dt.
\end{array}
\]

It is not difficult to check that
\[|I^{2,s}_{z_k, \delta}(u, x_{s_k})| \leq C \int_{\RR \setminus (-\delta,\delta)} 
\frac{1}{|t|^{1+2s}}dt = \frac{C}{s\delta^{2s}} ,\]
which implies, using that $c(s) \sim (1-s)$ as $s \nearrow 1$,
\[c(s)|I^{2,s}_{z_k, \delta}(u, x_{s_k})| \to 0, \qquad \text{as} \quad s\to 1.\]
From this previous estimate it is clear we must take limits first in $s$, and then in $\delta$. 

For the first integral, using a second order Taylor expansion of $\phi(x_{x_k}+tz_{k,\varepsilon})$ around $t=0$, we get 
\[I^{1,s}_{z_k, \delta}(\phi, x_{s_k})
= \frac{\delta^{2-2s}}{2(1-s)} \Big( \langle D^2\phi(x_{s_k}) z_{k}, z_{k} \rangle + o_\delta(1)\Big) \]
Hence, using the precise expression for $c(s)$ (that implies $c(s) \sim (1-s)$ as $s \nearrow 1$),
taking a subsequence such that $z_{k} \to \hat z$ for some $\hat z$, and using continuity of $\langle D^2\phi(x) z, z \rangle$ in both $x$ and $z$, we obtain
\[\lim_{\delta \to 0^+}\lim_{s\to 1^-}c(s)I^{1,s}_{z_k, \delta}(\phi, x_{s_k}) = \langle D^2\phi(\bar x) \hat z, \hat z \rangle. \]

Therefore, collecting the previous results, we obtain that
for any convergent sequence $z_{k} \to \hat z$, it holds that
 \[\lim_{\delta \to 0^+}\lim_{s\to 1^-}c(s)E^s_{z_{k},\delta} (u^g_s, \phi, x_{s_k}) = \langle D^2\phi(\bar x) \hat z, \hat z \rangle. \]

We would like to check that 
\[\langle D^2\phi(\bar x) \hat z, \hat z \rangle = \inf_{z \in \mathbb{S}^{N-1}} \langle D^2\phi(\bar x) z, z \rangle. \]
That is, 
\[ \langle D^2\phi(\bar x) \hat z, \hat z \rangle \leq \langle D^2\phi(\bar x) z, z \rangle, \quad \forall z \in \mathbb{S}^{N-1} . \]
We can get this last inequality taking limits in 
\begin{equation} \label{pepe}
E^s_{z,\delta} (u^g_s, \phi, x_s) \geq \inf_{z\in \mathbb{S}^{N-1}}E^s_{z,\delta} (u^g_s, \phi, x_s) \geq E_{z_{k},\delta} (u^g_s, \phi, x_s) - \frac{1}{k}.
 \end{equation}
 In fact, if we take the limit as $s \nearrow 1$ ($k \rightarrow \infty$) and as $\delta \searrow 0$ we have
  \[\lim_{\delta \to 0^+}\lim_{s\to 1^-}c(s)E^s_{z_{k},\delta} (u^g_s, \phi, x_{s_k}) = \langle D^2\phi(\bar x) \hat z, \hat z \rangle. \]
  and
   \[\lim_{\delta \to 0^+}\lim_{s\to 1^-}c(s)E^s_{z,\delta} (u^g_s, \phi, x_{s_k}) = \langle D^2\phi(\bar x) z, z \rangle. \]
   Then, from \eqref{pepe}, we obtain
   $$
   \langle D^2\phi(\bar x) \hat z, \hat z \rangle \leq \langle D^2\phi(\bar x) z, z \rangle,
   $$
as we wanted to show. 

For the supremum the proof is exactly the same, 
and hence we conclude that
$$
\begin{array}{l}
\displaystyle 
 \lim_{\delta \to 0^+}\lim_{s\to 1^-} c(s)
\Big\{ \inf_{z\in \mathbb{S}^{N-1}}E^s_{z,\delta} (u^g_s, \phi, x_s) 
+ \sup_{z\in \mathbb{S}^{N-1}}E^s_{z,\delta} (u^g_s, \phi, x_s)\Big\} \\[10pt]
\displaystyle \qquad =
 \inf_{z \in \mathbb{S}^{N-1}} \langle D^2\phi(\bar x) z, z \rangle
+ \sup_{z \in \mathbb{S}^{N-1}} \langle D^2\phi(\bar x) z, z \rangle
\\[10pt]
\displaystyle \qquad =\lambda_1 (D^2 \phi) (\bar x) + \lambda_N (D^2 \phi) (\bar x).
\end{array}
$$
As $u_s$ is a subsolution, we have 
$$
0\leq \Big\{ \inf_{z\in \mathbb{S}^{N-1}}E^s_{z,\delta} (u^g_s, \phi, x_s) 
+ \sup_{z\in \mathbb{S}^{N-1}}E^s_{z,\delta} (u^g_s, \phi, x_s)\Big\}
$$
and then in the limit we get that 
$$
0\leq \lambda_1 (D^2 \phi) (\bar x) + \lambda_N (D^2 \phi) (\bar x)
$$
showing that $\overline{u}$ is a viscosity subsolution.

Moreover, from Theorem \ref{condicion-borde.77} we obtain that
the half-relaxed limit satisfies
$$
\overline{u}^g \leq g
$$ 
on $\partial \Omega$ (we have a uniform barrier
at every boundary point that implies that $u_s \leq g$ on $\partial \Omega$). 

For the general case of $N$ eigenvalues, we just observe similar arguments to 
the previous ones, imply that the limit
\[\lim_{\delta \to 0^+}\lim_{s\to 1^-}c(s)E^s_{z_{k},\delta} (u^g_s, \phi, x_{s_k}) = \langle D^2\phi(\bar x) \hat z, \hat z \rangle \]
allows us to conclude that 
$$
\begin{array}{l}
\displaystyle
 \lim_{\delta \to 0^+}\lim_{s\to 1^-}\max_{{dim} (S) = N-i+1}
       \min_{ \substack{z \in S, \\ |z| =1}} c(s) E^s_{z,\delta} (u^g_s, \phi, x_s)
       \\[10pt] \qquad \displaystyle
       = \max_{{dim} (S) = N-i+1}\min_{ \substack{z \in S, \\ |z| =1}} 
       \langle D^2\phi(\bar x) z, z \rangle,
       \end{array}
$$
 from where the result follows as before.
\end{proof}

\begin{remark}
Analogously, we can obtain that the half-relaxed lower limit of supersolutions $u_s$, given by,
\[\underline{u}(x):= \liminf_{\subalign{s&\to 1^-\\y&\to x}}u_s(y), \] 
 is a supersolution 
 to the limit problem \eqref{Dirich-local-Lapla.sect}. 
 \end{remark}

Now, we are ready to prove the convergence result.

\begin{theorem}[Harmonic convergence] Let $u_s$ be a viscosity solution to our problem
for the trace fractional Laplacian,
\eqref{Main-eq-lim}. Then $u_s$ converge uniformly to the solution $u$ of the 
Dirichlet problem for the local Laplacian, \eqref{Dirich-local-Lapla.sect}.

It also holds that the solutions to the mid-range fractional Laplacian
converge to the solution to the limit problem \eqref{Dirich-local-Lapla.N=2.sect}.
\end{theorem}

\begin{proof}
    By the previous result, we know half-relaxed limits of viscosity subsolutions and supersolutions  converge respectively to subsolutions and supersolutions of the equivalent Laplacian problem. Using the definition of the half-relaxed limit and comparison, 
   we conclude that the half-relaxed limits coincide
   $$
   \liminf_{\subalign{s &\to 1^- \\ y &\to x}} u_s(y)
   = \limsup_{\subalign{s &\to 1^- \\ y &\to x}} u_s(y) =u(x)$$
   and is the solution to \eqref{Dirich-local-Lapla.sect}
   (since it is both a sub and a supersolution). 
   Hence, we have
   that the solutions $u_s$ converge to the solution of the limit problem in the following sense, 
    \begin{equation}
        \lim_{\subalign{s &\to 1^- \\ y &\to x}} u_s(y) = u(x).
    \end{equation}
    
    We only have to check that the convergence is uniform. If the convergence was not uniform, for any sequence $s_n \nearrow 1$ there exists some $\varepsilon_0>0$ and a corresponding sequence $\{x_{s_n}\} \subset \Omega$ such that
    \[|u_{s_n}(x_{s_n}) - u(x_{s_n})|\geq \varepsilon_0\]
    But in a compact set $\overline{\Omega}$ we can assume $x_{s_n}$ converges to some $x_0 \in \Omega$ after taking a subsequence. Due to its definition, $\lim_n u_{s_n}(x_{s_n}) = u(x_0)$. We arrive at a contradiction and the convergence had to be uniform.
\end{proof}

\begin{remark} 
Observe that, our previous arguments show that, 
for $\phi \in C^2(\RR^N)$ with sufficient decay at infinity, the fractional $s-$\textit{eigenvalues} 
given by \eqref{eq:eigenvalues.sect} (with the constant given by
\eqref{cte.99}) converge to the corresponding eigenvalues of the Hessian 
\[
    \lim_{s \to 1^-} \Lambda_k^s \phi(x) = \lambda_k (D^2\phi)(x)  
\] 
with 
\[\lambda_k (D^2 \phi)(x) := 
\max_{{dim} (S) = N-k+1}
       \min_{\substack{z \in S, \\ |z| =1}} \langle D^2 \phi(x) z, z \rangle. \]
\end{remark}

        \section{Possible extensions} \label{sect-extension}
        
        In this last section we briefly comment on possible extensions 
        of our results. 
        
        \subsection{A nontrivial righthand side}
        One can obtain similar existence, uniqueness and comparison results for
         \begin{equation} \label{Main-eq-section.446677} 
        \left\{
             \begin{array}{ll}
                \displaystyle 
                \sum_{i=1}^N  \Lambda_i^s u (x) = f(x) , \qquad & x \in \Omega, \\[6pt]
                         u (x) = g (x), \qquad & x \in \mathbb{R}^N\setminus \Omega,
             \end{array}
             \right.
         \end{equation}  
         as long as $f$ is continuous in $\overline{\Omega}$.  
         For Hölder regularity results for the solutions 
         when $g=0$ and $s$ is close to 1 we refer to \cite{HolderRef}.
        
         \subsection{Equations with coefficients}
       
         One can introduce $x-$dependent coefficients and study
          \begin{equation} \label{Main-eq-section.44667799} 
        \left\{
             \begin{array}{ll}
                \displaystyle \sum_{i=1}^N a_i (x) \Lambda_i^s u (x) =0, \qquad & x \in \Omega, \\[6pt]
                         u (x) = g (x), \qquad & x \in \mathbb{R}^N\setminus \Omega,
             \end{array}
             \right.
         \end{equation}    
         c.f. \eqref{Main-eq-section.4466}.
         Notice that the mid-range fractional Laplacian is of this form, just take $a_1=1$, $a_N=1$ and
         $a_i=0$ for $i=2,...,N-1$ in the expression 
         $\sum_{i=1}^N a_i (x) \Lambda_i^s u (x) $.
         
         Here, to use our previous arguments, we need that $a_i$ are continuous in $\overline{\Omega}$
         and nonnegative with $a_1$ and $a_N$
         strictly positive.

         Associated with this idea of introducing coefficients in our model problem we can 
         define fractional Pucci operators (described in terms 
         of the fractional eigenvalues) considering, for two real constants $0<\theta \leq \Theta $,
         $$
         P^+_{\theta, \Theta} (u) (x) = \Theta \sum_{\Lambda_i (u) (x) >0}  \Lambda_i^s u (x)
         + \theta \sum_{\Lambda_i (u) (x) < 0}  \Lambda_i^s u (x)
         =  \sup_{\theta \leq a_i \leq \Theta} \sum_{i=1}^N a_i  \Lambda_i^s u (x) 
         $$
and 
      $$
         P^-_{\theta, \Theta} (u) (x) = \theta \sum_{\Lambda_i (u) (x) >0}  \Lambda_i^s u (x)
         +  \Theta \sum_{\Lambda_i (u) (x) < 0} \Lambda_i^s u (x)
         =  \inf_{\theta \leq a_i \leq \Theta} \sum_{i=1}^N a_i  \Lambda_i^s u (x). 
         $$   
         These operators are extremal operators in the class of fractional trace Laplacians
         with coefficients between $\theta$ and $\Theta$. 
         
         Existence, uniqueness and a comparison principle for the Dirichlet problem
         for the operators $P^+_{\theta, \Theta} (u) $ and $P^-_{\theta, \Theta} (u)$
         in $C^2$ domains with continuous and bounded exterior data can be proved as
         in Section \ref{sect-compar}.

         \subsection{Operators defined by $\max$/$\min$ in subspaces}
         Instead of fractional eigenvalues we can consider 
         fractional truncated Laplacians to define our operator. 
         
        For instance, let us consider
         $$
         T^+_j (u) (x) = \sup_{dim(S) =j} \int_{S} \frac{u(x+y) - u(x)}{|y|^{j+2s}} dy
         $$
         and 
         $$
         T^-_i (u) (x) = \inf_{dim(S) =i} \int_{S} \frac{u(x+y) - u(x)}{|y|^{i+2s}} dy.
         $$
         Notice that here we are computing the supremum (or the infimum) of
         fractional Laplacians of the function $u$ restricted to 
         subspaces of dimension $j$ and hence the
         singularity of the kernel is of the form $|y|^{-j-2s}$. 
         
         Then, with these operators at hand one can define a
         different version of the fractional Laplacian, 
         $$
         (-\Delta)^s_{j,i} (u) (x) = -T^+_j (u) (x) - T^-_i (u) (x),
         $$
         for two indexes $j,i$ such that $j+i=N$.

         As examples, in $\mathbb{R}^4$, we can look at   
  $$
(-\Delta)^s_{2,2} (u) = -\inf_{dim(S) =2} \int_{S}
\frac{u(x+y) - u(x)}{|y|^{2+2s}} dy -
\sup_{dim(S) =2} \int_{S} \frac{u(x+y) - u(x)}{|y|^{2+2s}} dy
$$
or to
 $$
(-\Delta)^s_{3,1} (u) =- \inf_{dim(S) =1} \int_{S}
\frac{u(x+y) - u(x)}{|y|^{1+2s}} dy -
\sup_{dim(S) =3} \int_{S} \frac{u(x+y) - u(x)}{|y|^{3+2s}} dy.
$$

Remark that the same procedure with the maximum and minimum of the
local usual Laplacian acting on subspaces gives
the Laplacian in the whole space. In fact, we have
$$
\inf_{dim(S) = j}
\Delta u |_S (x) = \sum_{l=N-j}^N \lambda_l (D^2 u) (x) 
$$
and 
$$
\inf_{dim(S) = i}
\Delta u |_S (x) = \sum_{l=1}^i \lambda_l (D^2 u) (x), 
$$
that is, the supremum over subspaces $S$ of dimension $j$ 
os the Laplacians of $u$ restricted to $S$ is given by
the sum of the $j$ largest eigenvalues of $D^2u(x)$ and similarly
the infimum among subspaces of dimension $i$ is the sum
of the smallest eigenvalues of $D^2u(x)$. 
Therefore, for any pair $(j,i)$ such that $j+i=N$ we have
$$
\inf_{dim(S) = i}
\Delta u |_S (x)
+ \sup_{dim(S) = j}
\Delta u |_S (x)= \Delta u (x).
$$

However, in the fractional setting the operators
$(-\Delta)^s_{2,2} (u)$ and $(-\Delta)^s_{3,1} (u)$ are {\it different}.

Finally, let us point out that we can also study 
$\max$/$\min$ procedures. For example, one
can consider, for $j\geq i$,
$$
         W^+_{j,i} (u) (x) = \sup_{{dim} (S) = j}
       \inf_{ \substack{T \subset S, \\ dim(T) =i }} 
        \int_{T} \frac{u(x+y) - u(x)}{|y|^{i+2s}} dy.
       $$
       Here we are computing the supremum among 
       subspaces $S$ of dimension $j$ of
       the infimum of fractional Laplacians (of dimension $i$) 
       of $u$ restricted to subspaces $T$ included in $S$.  
       Notice that the fractional eigenvalues $\Lambda_j (u) (x)$
       that we used here to define the trace fractional Laplacian
       are given by
       $$
       \Lambda_k^s u (x) = \max_{{dim} (S) = N-k+1}
       \min_{ \substack{z \in S, \\ |z| =1} }  \int_{\mathbb{R}} 
        \frac{u(x+tz)-u(x)}{|t|^{1+2s}} 
        \, dt,
    $$
    that corresponds to the previous setting 
    taking $j=N-k+1$ and $i=1$. 
    
    With these operators $W^+_{j,i} (u)$ we can
    obtain a fractional version of the Laplacian adding 
    them (taking care of the fact that the dimensions 
    of the corresponding subspaces add up to $N$), that is,
    just consider
    $$
    (-\widetilde{\Delta})_{(j_1,i_1)...(j_k,i_k)}^s (u) (x) =
    - \sum_{l=1}^k W^+_{j_l,i_l} (u) (x).
    $$
It should be interesting to know if there is a comparison 
principle for operators like this.

{\bf Acknowledgements.} J. D. Rossi was partially supported by 
			CONICET PIP GI No 11220150100036CO
(Argentina), PICT-2018-03183 (Argentina) and UBACyT 20020160100155BA (Argentina).

J. Ruiz-Cases was supported by the European Union's Horizon 2020
research and innovation programme under the Marie Sklodowska-Curie grant
agreement No.\,777822, and by grants
CEX2019-000904-S, PID2020-116949GB-I00, and RED-2022-134784-T, all three
funded by MCIN/AEI/10.13039/501100011033 (Spain).


     \end{document}